\documentclass{amsart}

\usepackage{amsrefs}

\newtheorem{theorem}{Theorem}[section]
\newtheorem{proping}{Proposition}[section]
\newtheorem{coring}{Corollary}[section]
\newtheorem{lemma}[theorem]{Lemma}

\newtheorem{definition}[theorem]{Definition}
\newtheorem{remark}[theorem]{Remark}

\begin{document}                                                 
    \title[Solutions to the singular $\sigma_2-$Yamabe problem]{Solutions to the singular $\sigma_2-$Yamabe problem with isolated singularities}
    \author[Almir Silva Santos]{Almir Silva Santos}
    \address{Universidade Federal de Sergipe\\
Centro de Ci\^encias Exatas e Tecnologia\\
Departamento de Matem\'atica\\
\;Av. Marechal Rondon s/n\\
S\~ao Crist\'ov\~ao-SE, 49100-00\\}
    \email{arss@ufs.br}
    \thanks{The author would like to thank the support provided by CNPq-Brazil and Institut Fourier-France where this work was started during his Post-doct. He is specially grateful to Fernanco C. Marques for his constant encouragement.}
    \keywords{singular $\sigma_k-$Yamabe problem, constant $\sigma_k-$curvature, Weyl tensor, gluing method.}
    \subjclass{53C21; 53A30}
    \begin{abstract}
    Given $(M,g_0)$ a closed Riemannian manifold and a nonempty closed subset $X$ in $M$, the singular $\sigma_k-$Yamabe problem asks for a complete metric $g$ on $M\backslash X$ conformal to $g_0$ with constant $\sigma_k-$curvature. The $\sigma_k-$curvature is defined as the $k-$th elementary symmetric function of the eigenvalues of the Schouten tensor of a Riemannian metric. The main goal of this paper is to find solutions to the singular $\sigma_2-$Yamabe problem with isolated singularities in any compact non-degenerate manifold such that the Weyl tensor vanishing to sufficiently high order at the singular point. We will use perturbation techniques and gluing methods.
    \end{abstract}                
    \maketitle

    
\section{Introduction}

Since the complete resolution of the Yamabe problem by Yamabe \cite{MR0125546}, Trudinger \cite{MR0240748}, Aubin \cite{MR0431287} and Schoen \cite{MR788292}, much attention has been given to the study of conformal geometry. To understand the problem we are interested in this work, first lets recall some background definition from Riemannian Geometry. 
Given a Riemannian manifold $(M,g)$, there exists an orthogonal decomposition of the curvature tensor $Rm_g$ which is given by
$$Rm_g=W_g+A_g\odot g,$$
where $\odot$ is the Kulkarni-Nomizu produt, $A_g$ is the Shouten Tensor defined as
\begin{equation}\label{eq056}
 A_g=\frac{1}{n-2}\left(Ric_g-\frac{1}{2(n-1)}R_gg\right),
\end{equation}
$Ric_g$ and $R_g$ are respectively the Ricci tensor and the scalar curvature of the metric $g$, see \cite{MR2274812} for instance. Since the Weyl curvature tensor $W_g$ is conformally invariant in the sense that $W_{e^fg}=e^fW_g$, then to understant the conformal class of the metric $g$ it is natural to study the Schouten tensor $A_g$. For $k\in\{1,\ldots,n\}$, the $\sigma_k-$curvature is defined as
$$\sigma_k(A_g):=\sum_{1\leq i_1<\cdots< i_k\leq n}\lambda_{i_1}\cdot\ldots\cdot\lambda_{i_k},$$
that is, $\sigma_k(A_g)$ is the $k-$th elementary symmetric function of the eigenvalue $(\lambda_1,$ $\ldots,$ $\lambda_n)$ of $A_g$.

The $\sigma_k-$Yamabe problem asks for a conformal metric in a given closed Riemannian manifold $(M,g)$ with $\sigma_k-$curvature constant. Note that since $\sigma_1(A_g)=\frac{1}{2(n-1)}R_g$, then the case $k=1$ is the classical Yamabe problem. The $\sigma_k-$Yamabe problem has been extensively studied in the past years. We direct the reader to the papers \cite{MR1945280}, \cite{MR1923964}, \cite{MR1978409}, \cite{MR2373147}, \cite{MR1988895}, \cite{MR2362323} and the references contained therein.

It is then natural to ask whether every noncompact Riemannian manifold is con\nolinebreak formally equivalent to a complete manifold with constant $\sigma_k-$curvature. When the noncompact manifold has a simple structure at infinity, this question may be studied by solving the singular $\sigma_k-$Yamabe problem: Given $(M,g_0)$ a closed Rie\nolinebreak mannian manifold and a nonempty closed set $X$ in $M$, find a complete metric $g$ on $M\backslash X$ conformal to $g_0$ with constant $\sigma_k-$curvature. For $k=1$ this problem has been extensively studied in recent years, and many existence results as well as obstructions to existence are known. See \cite{MR2639545} and the references contained therein (See also \cite{BPS}).

The equation $\sigma_k(A_g)=$ constant is always elliptic for $k=1$, while for $k\geq 2$ we need some additional hypothesis, for example, a sufficient condition for this is that $g$ is $k-$admissible. By definition a metric $g$ on $M$ is said to be $k-$admissible if it belongs to the $k-$th positive cone $\Gamma_k^+$, this means that 
$$g\in\Gamma_k^+\Longleftrightarrow\sigma_1(A_g),\ldots,\sigma_k(A_g)>0.$$

We will produce conformal complete metric in a given closed manifold $(M,g_0)$ with nonremovable isolated singularity with positive constant $\sigma_2-$curvature. Before write precisely our main result lets remember some well known facts about the $\sigma_k-$curvature.

For $4\leq 2k<n$ it was proved in \cite{MR2040327} and \cite{MR2169873} that if $\mathbb S^n\backslash X$ admits a complete Riemannian metric $g$ conformal to the round metric $g_{\mathbb S^n}$ with $\sigma_1(A_g)\geq c>0$ and $\sigma_2(A_g),\ldots,\sigma_k(A_g)\geq 0$, then the Hausdorff dimension of $X$ is less then or equal to $(n-2k)/2$. On the other hand, using the estimate, obtained in \cite{MR1938739}, namely,
$$Ric_g\geq \frac{(2k-n)(n-1)}{(k-1)}\left(\begin{array}{c}
                                            n\\
                                            k
                                           \end{array}
\right)^{-1/k}\sigma_k^{1/k}(A_g)g,$$
for locally conformally flat manifold and the Bonnet-Myers's theorem, Gonzalez \cite{MR2169873} observed that there is no singular metric in $\mathbb S^n$ with positive constant $\sigma_k-$curvature and $n<2k$. In \cite{MR2165306} the authors proved that there is no complete metric with constant positive $\sigma_{n/2}-$curvature, conformally related with the canonical metric in $\mathbb R^n\backslash\{0\}$, and with radial conformal factor. For $2\leq k\leq n$, Han, Li and Teixeira \cite{MR2737708} proved, as in the case $k=1$ (see \cite{MR982351}, \cite{MR1666838} and \cite{MR2393072}), that any complete metric in $\mathbb S^n$ with nonremovable isolated singularity, positive constant $\sigma_k-$curvature and conformal to the canonincal one is asymptotic to some rotationally symmetric metric near the singular set. Although these results are for locally conformally flat manifold, they motivate us to consider the singular $\sigma_k-$Yamabe problem with $2\leq 2k<n$.

In \cite{MR2393072}, Marques proved that given a closed manifold $(M,g)$, not necessarilly locally conformally flat, with dimension $3\leq n\leq 5$ then every complete metric with positive constant scalar curvature and with nonremovable isolated singularity is asymptotic to a radial one near the singular set. It should be an interesting question ask if there is an analogous result for singular metrics with positive constant $\sigma_k-$curvature. Another interested problem is related with the Hausdorff dimension estimate $(n-2k)/2$. For $k=1$ this estimate is sharp. In \cite{MR1712628}, the authors constructed metrics of constant positive scalar curvature that are singular at any given disjoint union of smooth submanifolds of $\mathbb S^n$ of dimensions $0<k_i\leq(n-2)/2$. In fact, a model to the positive singular Yamabe problem is the manifold $\mathbb S^{n-l-1}\times\mathbb H^{l+1}$ which is conformal to $\mathbb S^n\backslash\mathbb S^l$ and has positive constant scalar curvature equal to $(n-2l-2)(n-1)$ for all $l<(n-2)/2$, see \cite{BPS} and \cite{MR2169873}. Up to our knowledge, it is not known if the correspond estimate for $k>1$ is sharp. Gonzalez \cite{MR2169873} has showed that 
$$l_k:=\sup\{l\geq 0;P_1(l),\ldots,P_k(l)>0\}\rightarrow\frac{n-2}{2}-O(\sqrt{n}),\;\;\mbox{ as }\;\;n\rightarrow\infty,$$
where $P_r(l)$ is the $\sigma_r-$curvature of $\mathbb S^{n-l-1}\times\mathbb H^{l+1}$. See \cite{MR3237067} for more details about this subject.

Only few results are known about the singular $\sigma_k-$Yamabe problem. Using a similar method like Mazzeo and Pacard \cite{MR1712628} used to construct singular metrics in the sphere $\mathbb S^n$ with positive constant scalar curvature, Mazzieri and Ndiaye \cite{MN} proved the following existence result:
\begin{theorem}[Mazzieri-Ndiaye \cite{MN}]
 Suposse $\Lambda\subset\mathbb S^n$ is a finite set which is symmetrically balanced, that is, there exists an orthogonal transformation $T\in O(n+1)$ of $\mathbb R^{n+1}$ such that $T(\Lambda)=\Lambda$ and 1 is not an eigenvalue of $T$. Assuming $2\leq 2k<n$, then there exists a family of complete Riemannian metric on $\mathbb S^n\backslash\Lambda$ with positive constant $\sigma_k$-curvature, which are conformal to the standard metric in $\mathbb S^n$.
\end{theorem}

We notice that by a result in \cite{MR2165306} there is no complete metric in $\mathbb S^n\backslash\{p\}$ with posi-tive constant $\sigma_k-$curvature which is radially simmetric and conformal to the standard round metric. If $\Lambda=\{p_1,\ldots,p_m\}\subset\mathbb S^n$ is a finite set which is symmetrically balanced, then $T(p)=p$, where $p=\sum p_i$ and $T\in O(n+1)$ is a linear orthogonal transformation such that 1 is not an eigenvalue of $T$ and $\Lambda$ is $T-$invariant. This implies that the only possibility is that $p=0$, and so $m\geq 2$.

Inspired by the construction presented in \cite{MR1399537}, Mazzieri and Segatti \cite{MR2900438} has constructed complete noncompact locally conformally flat metrics with positive constant $\sigma_k-$curvature with $2\leq k<n$. The method consists in performing the connected sum of a finite number of Delaunay-type metrics. For connected sum in the compact case see \cite{MR3023858}.

Our main result is concerned with the positive singular $\sigma_2-$Yamabe problem in the case where $X$ is finite set which can be a single point. We will construct solutions to this problem under a condition on the Weyl tensor. The method which we will apply is based on perturbation techniques and gluing procedure. This method was applied in \cite{MR2639545} to solve the problem in the case $k=1$. We restrict ourselves to the case $k=2$ since by the identity $\sigma_2(A_g)=\frac{1}{2}((\mbox{tr}_gA_g)^2-|A_g|^2_g)$ we find explicitly an expression to the equation $\sigma_2(A_g)=$ constant. It is an interesting problem to solve the singular $\sigma_k-$Yamabe problem for $3\leq k< n/2$. 

The main result of this paper reads as follows:
\vspace{0,5cm}

\noindent{\bf Main Theorem:} {\it Let $(M^n,g_0)$ be an compact Riemannian manifold nondege-nerate with dimension $n\geq 5$, $g_0$ conformal to some $2-$admissible metric and the $\sigma_2-$curvature is equal to $n(n-1)/8$. Let ${\{p_1,\ldots,p_m\}}$ be a set of points in $M$ such that $\nabla_{g_0}^jW_{g_0}(p_i)=0$ for $j=0,1,\ldots,\left[\frac{n-4}{2}\right]$ and $i=1,\ldots,m$, where $W_{g_0}$ is the Weyl tensor of the metric $g_0$. Then, there exist a constant $\varepsilon_0>0$ and a one-parameter family of complete metrics $g_\varepsilon$ on $M\backslash\{p_1,\ldots,p_m\}$ defined for $\varepsilon\in(0,\varepsilon_0)$, conformal to $g_0$, with constant $\sigma_2-$curvature equal to $n(n-1)/8$, ob\nolinebreak tained by at\nolinebreak taching Delaunay-type ends to the points $p_1,\ldots,p_m$. Moreover, $g_\varepsilon\rightarrow g_0$ uniformly on compact sets in $M\backslash\{p_1,\ldots,p_m\}$ as $\varepsilon\rightarrow 0$.}\\

We notice here that by a result of Sheng, Trudinger and Wang \cite{MR2362323} for all $4\leq 2k\leq n$ the positive $\sigma_k-$Yamabe problem always has a solution since the operator is variational and the initial metric is conformal to another one $k-$admissible. But, by a result due to Viaclovsky \cite{MR1738176} the operator is always variational in the case $k=2$. Also, it is well known by \cite{MR2389992} that for $3\leq k\leq n$ the operator is variational if and only if the manifold is locally conformally flat. 


The nondegeneracy is defined as follows
\begin{definition}\label{def5}
 A metric $g$ with constant $\sigma_2-$curvature equal to $n(n-1)/8$ is {\it nondegenerate} if the operator $L_g^1:C^{2,\alpha}(M)\rightarrow C^{0,\alpha}(M)$ is surjective for some $\alpha\in(0,1)$, where $L_g^1$ is defined in (\ref{eq030}). Here $C^{k,\alpha}(M)$ is the standard H\"older spaces on $M$.
\end{definition}

When the operator $L_g^1$ is elliptic, we need only check the injectivity. For example, it is clear that the round sphere $\mathbb S^n$ is degenerate because $L_{g_0}^1=c_n(\Delta_{g_0}+n)$ annihilates the restrictions of linear functions on $\mathbb R^{n+1}$ to $\mathbb S^n$, where $c_n=-(n-1)(n-4)/8$.

Mazzieri and Ndiaye proved their theorem in the sphere, which is locally conformally flat. With this assumption, in the neighborhood of $p_i$ the metric is esentially the standard metric on $\mathbb R^n$, and in this case it is possible to transfer the metric to cylindrical coordinates, where there is a family of well known Delaunay-type solutions. In our case we only have that the Weyl tensor vanishes to sufficiently high order at each point $p_i$. Since the singular $\sigma_k-$Yamabe problem is conformally invariant, it is more convenient to work in conformal normal coordinates. As indicated in \cite{MR2477893} in such coordinates we get some simplifications. The order $\left[\frac{n-4}{2}\right]$ comes up naturally in our method and will be fundamental to solve the problem locally, although we do not know if it is the optimal one.

The organization of this paper is as follows. 

All the analysis in the paper are done considering $m=1$. In Section \ref{sec03} we record some notation that will be used throughout the work. We review some results concerning the Delaunay-type solutions for the constant $\sigma_k-$curvature equation and using the right inverse found in \cite{MN} and a perturbation argument we construct a right inverse for the linearized operator about such solution. In Section \ref{sec02} we work with conformal normal coordinates in a neighborhood of $p$, since in theses coordinates we get some simplifications. We use the assumption on the Weyl tensor to finding a family of constant $\sigma_2-$curvature metrics in a small ball centered at the singular point, which depends on $n+2$ parameters with prescribed Dirichlet data. In Section \ref{sec04} we work with a metric which has constant $\sigma_2-$curvature and we find a family of constant $\sigma_2-$curvature metrics, which also depends on $n+2$ parameters with prescribed Dirichlet data. Finally, in Section \ref{sec05} we put all results obtained in the previous sections together to proof the Main Theorem for the case $m=1$. The fact that the metric is conformal to some one $2-$admissible allow us to use elliptic regularity. For the general case we briefly explain the minor changes that need to be made in order to deal with more than one singular point.


\section{Preliminaries}\label{sec03}

In this section we record some notation and results that will be used frequently, throughout the rest of the work and sometimes without comment. 

We use the symbols $c$, $C$, with or without subscript, to denote various positive constants. We write $f=O'(Cr^k)$ to mean $f=O(Cr^{k})$ and $\nabla f=O(Cr^{k-1})$. $O''$ is defined similarly.

\subsection{Notation}
Let us denote by $e_j$, for $j\in\mathbb{N}$, the eigenfunction of the Laplace operator on $\mathbb{S}^{n-1}$ with corresponding eigenvalue $\lambda_j$, where $\lambda_0=0$, $\lambda_1=\cdots=\lambda_n=n-1$, $\lambda_{n+1}=2n, \ldots$ and $\lambda_j\leq\lambda_{j+1}$ with unit $L^2-$norm. That is, 
$$\Delta_{\mathbb{S}^{n-1}}e_j+\lambda_j e_j=0\;\;\;\mbox{ and }\;\;\;\|e_j\|^2_2=\int_{\mathbb S^{n-1}}e_j^2=1.$$

Remember that $\{e_j\}$ is an orthonormal basis of $L^2(\mathbb S^{n-1})$. These eigenfunctions are restrictions to $\mathbb{S}^{n-1}\subset\mathbb{R}^n$ of homogeneous harmonic polynomials in $\mathbb{R}^n$. The $i-$th eigenvalue counted without multiplicity is $i(i+n-2)$.

Let $\mathbb{S}^{n-1}_r$ be the sphere with radius $r>0$. If the eigenfunction decomposition of the function $\phi\in L^2(\mathbb{S}_r^{n-1})$ is given by
$$\phi(r\theta)=\sum_{j=0}^\infty \phi_j(r)e_j(\theta)\;\;\;\mbox{ where }\;\;\;\phi_j(r)=\int_{\mathbb{S}^{n-1}}\phi(r\cdot)e_j,$$
then we define the projection $\pi_r''$ onto the {\it high frequencies space} by the formula
$$\pi''_r(\phi)(r\theta):=\sum_{j=n+1}^\infty\phi_j(r)e_j(\theta).$$
The {\it low frequencies space} on $\mathbb{S}_r^{n-1}$ is spanned by the constant functions and the restrictions to $\mathbb{S}^{n-1}_r$  of linear functions on $\mathbb{R}^n$. 

\subsection{The constant $\sigma_k-$curvatura equation}

Let $(M,g_0)$ be a closed Riemannian manifold of dimension $n\geq 3$. Let $A_{g_0}$ be the Schouten tensor of the metric $g_0$ defined in (\ref{eq056}).

The so-called $\sigma_k-$curvatura of $(M,g_0)$, which is a smooth function denoted by $\sigma_k(A_{g_0})$, is defined pointwise for each $p\in M$ as the $k-$th symmetric elementary function of the eigenvalues of the tensor $A_{g_0}(p)$. Since
$$\sigma_1(A_{g_0})=\mbox{tr}_{g_0}(A_{g_0})\;\;\mbox{ and }\;\;\sigma_2(A_{g_0})=\frac{1}{2}\left(\mbox{tr}_{g_0}(A_{g_0})^2-|A_{g_0}|_{g_0}^2\right),$$
notice that
\begin{equation}\label{eq010}
\sigma_1(A_{g_0})=\frac{R_{g_0}}{2(n-1)}\;\;\;\;\mbox{ and }\;\;\;\;\sigma_2(A_{g_0})=\frac{n}{8(n-1)(n-2)^2}R_{g_0}^2- \frac{|Ric_{g_0}|_{g_0}^2}{2(n-2)^2}.
\end{equation}
The Euclidean space $\mathbb{R}^n$ with its standard metric is $\sigma_k-$flat for any $1\leq k\leq n$, whereas the standard sphere $\mathbb{S}^n$ has $A_{\mathbb{S}^n}=\frac{1}{2}g_{\mathbb{S}^n}$ and thus 
$$\sigma_k(A_{\mathbb{S}^n})=\frac{1}{2^k}\left(\begin{array}{c}
n\\
k
\end{array}\right)\;\;\mbox{ for } \;\;1\leq k\leq n.$$

For a given nonempty closet set $X\subset M$, the positive singular $\sigma_k-$Yamabe problem amounts to find a conformal factor $u\in C^\infty(M\backslash X)$ such that the metric $g=u^{\frac{4k}{n-2k}}g_0$ is complete on $M\backslash X$ and verifies
\begin{equation}\label{eq000}
\sigma_k(A_{g})=\frac{1}{2^{k}}\left(\begin{array}{c}
n\\
k
\end{array}\right)\;\;\;\mbox{ in }\;\;\;M\backslash X.
\end{equation}

Now we define the nonlinear operator
\begin{equation}\label{eq057}
H_{g_0}(u)=\left(\frac{n-2k}{4k}\right)^ku^{\frac{2kn}{n-2k}}\sigma_k(A_{g})-\left(\begin{array}{c}
n\\
k
\end{array}\right)\left(\frac{n-2k}{4k}\right)^k u^{\frac{2kn}{n-2k}}. 
\end{equation}

The equation (\ref{eq000}) is equivalent to
\begin{equation}\label{eq015}
H_{g_0}(u)=0\;\;\;\mbox{ in }\;\;\;M\backslash X,
\end{equation}
with a suitable condition in the singular set, for instance, the function $u$ goes to infinity with a sufficiently fast rate. This equation is fully nonlinear for $k>1$.

The operator $H_{g_0}$ obeys the following relation concerning conformal changes of the metric
\begin{equation}\label{eq007}
H_{v^{4k/(n-2k)}g}(u)=v^{-\frac{2kn}{n-2k}}H_{g}(vu).
\end{equation}
and the Schouten tensor obeys the following well transformation law
$$A_{v^{4k/(n-2k)}g}=A_{g}-\frac{2k}{n-2k}u^{-1}\nabla_{g}^2u+\frac{2kn}{(n-2k)^2}u^{-2}du\otimes du-\frac{2k^2}{(n-2k)^2}u^{-2}|du|_{g}^2g$$

In this work we are interested in the case $k=2$. So, using the second formula in (\ref{eq010}) we obtain the expression for the nonlinear operator $H_{g_0}$ in this case
\begin{equation}
\begin{array}{rcl}\label{eq014}
H_{g_0}(u) & = & \left(\displaystyle\frac{n-4}{4}\right)^2u^4\sigma_2(A_{g_0}) +\displaystyle \frac{u^2}{2}(\Delta_{g_0} u)^2 -\displaystyle\frac{n-4}{8(n-2)}R_{g_0}u^2|\nabla_{g_0} u|_{g_0}^2 \\
\\
& - & \displaystyle\frac{n-4}{8(n-2)}R_{g_0}u^3\Delta_{g_0} u+ \displaystyle \frac{n-2}{n-4}u|\nabla_{g_0} u|_{g_0}^2\Delta_g u- \displaystyle\frac{u^2}{2}|\nabla_{g_0}^2u|_{g_0}^2\\
\\
& + & \displaystyle\left\langle Ric_{g_0},\frac{n-4}{4(n-2)}u^3\nabla_{g_0}^2u-\frac{n}{4(n-2)}u^2\nabla_{g_0} u\otimes\nabla_{g_0} u\right\rangle_{g_0} \\
\\
& + & \displaystyle\frac{n}{n-4}u\langle\nabla_{g_0} u\otimes\nabla_{g_0} u,\nabla_{g_0}^2u\rangle_{g_0}- \displaystyle\frac{n(n-1)(n-4)^2}{128}|u|^{\frac{3n+4}{n-4}}u.
\end{array}
\end{equation}

We seek a positive function which solves (\ref{eq015}). We will use perturbation techniques and gluing methods to finding this solution. Expanding $H_{g}$ about a function $u$, not necessarilly a solution, gives
$$H_g(u+v)=H_g(u)+L_g^{u}(v)+Q^{u}(v),$$
where
\begin{equation}\label{eq037}
\begin{array}{rcl}
   L_g^{u}(v) & = & \displaystyle\left.\frac{d}{dt}\right|_{t=0}H_g(u+tv)\\
   \\
   & = & \left(\displaystyle u^2\Delta u- \frac{n-4}{8(n-2)}R_gu^3+\frac{n-2}{n-4}u|\nabla u|^2\right)\Delta v\\
   \\
   & + & \left(\displaystyle \frac{(n-4)^2}{4}u^3\sigma_2(A_g)+u(\Delta u)^2-\frac{n-4}{4(n-2)}R_gu|\nabla u|^2\right.
   \end{array} 
\end{equation}
$$\begin{array}{rcl}
& + & \displaystyle\frac{n-2}{n-4}|\nabla u|^2\Delta u -\frac{3(n-4)}{8(n-2)}R_gu^2\Delta u -u|\nabla^2u|^2\\
\\
  & + & \displaystyle\frac{3(n-4)}{4(n-2)}u^2\langle Ric_g,\nabla^2u\rangle-\frac{n}{2(n-2)}u\langle Ric_g,\nabla u\otimes\nabla u\rangle\\
   \\
  & + & \left.\displaystyle\frac{n}{n-4}\langle\nabla u\otimes\nabla u,\nabla^2u\rangle-\frac{n^2(n-1)(n-4)}{32}|u|^{\frac{3n+4}{n-4}}\right)v\\
  \\
& + & \displaystyle\left\langle \frac{2(n-2)}{n-4}u\nabla u\Delta u-\frac{n-4}{4(n-2)}R_gu^2\nabla u,\nabla v\right\rangle\\
\\
& + & \left\langle \displaystyle\frac{n-4}{4(n-2)}u^3Ric_g-u^2\nabla^2u+\frac{n}{n-4}u\nabla u\otimes\nabla u,\nabla^2v\right\rangle\\
\\
& + & \displaystyle\left\langle \frac{2n}{n-4}u\nabla^2u-\frac{n}{2(n-2)}u^2Ric_g,\nabla u\otimes\nabla v\right\rangle
  \end{array} 
$$
and
\begin{equation}\label{eq039}
 Q^u(v)=\int_0^1\int_0^1\frac{d}{ds}L_g^{u+tsv}(v)dsdt.
\end{equation}

Note that, from the property (\ref{eq007}), we obtain
\begin{equation}\label{eq024}
 L_{u^{4k/(n-2k)} g}^v(w)=u^{-\frac{2kn}{n-2k}}L_g^{uv}(uw).
\end{equation}

It is important to emphasize here that in this work $(M,g_0)$ always will be a compact Riemannian manifold of dimension $n\geq 5$ with constant $\sigma_2-$curvature equal to $n(n-1)/8$ and nondegenerate, see Definition \ref{def5}. This implies that the operator $L_{g_0}^1:C^{2,\alpha}(M)\rightarrow C^{0,\alpha}(M)$ given by
\begin{equation}\label{eq030}
 L_{g_0}^1u=-\frac{n-4}{8(n-2)}R_{g_0} \Delta_{g_0}u- \frac{n(n-1)(n-4)}{8}u+\frac{n-4}{4(n-2)}\langle Ric_{g_0},\nabla_{g_0}^2u\rangle_{g_0}
\end{equation}
is surjective for some $\alpha\in(0,1)$. In the round sphere $\mathbb S^n$ we have $R_{g_0}=n(n-1)$ and $Ric_{g_0}=(n-1)g_0$, so
$$L_{g_0}^1u=-\frac{(n-1)(n-4)}{8}(\Delta_{g_0}+n)u.$$

\subsection{Delaunay-type solutions}\label{sec08}

In this section we recall some facts about the Delaunay-type solutions in the $\sigma_k-$curvature setting. Our solution to the singular $\sigma_k-$Yamabe problem will be asymptotic to some Delaunay-type solutions.

If $g=u^{\frac{4k}{n-2k}}\delta$ is a complete metric in $\mathbb{R}^n\backslash\{0\}$ conformal to the Euclidean standard metric $\delta$ on $\mathbb{R}^n$ with constant $\sigma_k-$curvature equal to $2^{-k}\left(\begin{array}{c}
n\\
k
\end{array}\right)$, then $u$ is a solution of the equation 
\begin{equation}\label{eq006}
H_\delta(u)=0\;\;\mbox{ in }\;\;\mathbb{R}^n\backslash\{0\}.
\end{equation}

Let us consider that $u$ is rotationally invariant, and thus the equation it satisfies may be reduced to an ordinary differential equation. These metrics has been studied in \cite{MR2165306}, see also \cite{MN}.

Since $\mathbb{R}^n\backslash\{0\}$ is conformally diffeomorphic to a cylinder, it will be  convenient to use the cylindrical background. In other words, consider the conformal diffeomorphism
$\Phi:(\mathbb{S}^{n-1}\times\mathbb{R},g_{cyl})\rightarrow (\mathbb{R}^n\backslash\{0\},\delta)$
defined by $\Phi(\theta,t)=e^{-t}\theta$ and where $g_{cyl}:=d\theta^2+dt^2$. Then $\Phi^*\delta=e^{-2t}g_{cyl}$ . Define $v(t):=e^{\frac{2k-n}{2k}t}u(e^{-t}\theta)=|x|^{\frac{n-2k}{2k}}u(x)$, where $t=-\log |x|$ and $\theta=\frac{x}{|x|}$. Note that $v$ is defined in the whole cylinder and $\Phi^*g= v^{\frac{4k}{n-2k}}g_{cyl}.$

Therefore, the conformal factor $v$ satisfies the following ODE
\begin{equation}\label{eq001}
\left(v^2-\left(\frac{2k}{n-2k}\right)^2\dot{v}^2\right)^{k-1}\left(v-\left( \frac{2k}{n-2k}\right)^2\ddot{v}\right)=\frac{n}{n-2k}v^{\frac{2kn}{n-2k}-1}.
\end{equation}

The Hamiltonian energy, given by
\begin{equation}\label{002}
H(v,w)=\left(v^2-\left(\frac{2k}{n-2k}\right)^2w^2\right)^{k}-v^{\frac{2kn}{n-2k}}
\end{equation}
is constant along solutions of (\ref{eq001}). We summarize the basic properties of this solutions in the next proposition (see Propositon 2.1 in \cite{MN} and Proposition 3.1 in \cite{MR2900438}, see also \cite{MR2165306}).
\begin{proping}\label{propo000}
Suppose $H(v,\dot{v})=H_0\in\left[0,\frac{2k}{n-2k} \left(\frac{n-2k}{n}\right)^{\frac{n}{2k}}\right]$, then we have three cases:
\begin{enumerate}
\item[a) ] If $H_0=0$, then either we have the trivial solution $v\equiv 0$, or $v(t)=\cosh^{-\frac{n-2k}{2k}}(t-c)$, for some $c\in\mathbb{R}$. The latter conformal factor gives rise to a metric on $\mathbb{S}^{n-1}\times\mathbb{R}$ which is non complete and which corresponds in fact to the standard metric $g_{\mathbb{S}^n}$ on $\mathbb{S}^n\backslash\{p,-p\}$.

\item[b) ] If $0<H_0<\frac{2k}{n-2k}\left(\frac{n-2k}{2k}\right)^{\frac{n}{2k}}$, then, in correspondence of each $H_0$, there exists a unique solution $v$ of (\ref{eq001}) satisfying the conditions $\dot{v}(0)=0$, and $\ddot{v}(0)>0$. This solution is periodic and it is such that $0<v(t)<1$ for all $t\in\mathbb{R}$. This family of solutions gives rise to a family of complete and periodic metrics on $\mathbb{R}\times\mathbb{S}^{n-1}$. 

\item[c) ] If $H_0=\frac{2k}{n-2k}\left(\frac{n-2k}{n}\right)^{\frac{n}{2k}}$, then there exists a unique solutions to (\ref{eq001}) given by $v(t)=\left(\frac{n-2k}{n}\right)^{\frac{n-2k}{4k^2}}$, for $t\in\mathbb{R}$. This solution give rise to a complete metric on $\mathbb{S}^{n-1}\times\mathbb{R}$ and it is in fact a constant multiple of the cilindrical metric $g_{cyl}$.

\end{enumerate}
\end{proping}
We will write the solution of (\ref{eq001}) given by the Proposition \ref{propo000} when $H_0>0$ as $v_\varepsilon$, where $v_{\varepsilon}(0)=\min v_{\varepsilon}=\varepsilon^{(n-2k)/2k}$, for $\varepsilon\in(0,((n-2k)/n)^{\frac{1}{2k}})$ and the corresponding solution of (\ref{eq006}) as $u_\varepsilon(x)=|x|^{\frac{2k-n}{2k}}v_\varepsilon(-\log|x|)$.

For our  purposes, the next proposition gives sufficient information about their behavior as $\varepsilon$ tends to zero. Its proof can be found in \cite{MN}, but we include it here for the sake of the reader.
\begin{proping}\label{propo003}
For $0<\varepsilon<\left(\frac{n-2k}{n}\right)^{\frac{1}{2k}}$. Then we have that there exists a positive constant $c_{n,k}>0$ depending only on $n$ and $k$ such that for all $t\in\mathbb R$ we have
$$\begin{array}{rcr}
   \left|v_\varepsilon(t)-\varepsilon^{\frac{n-2k}{2k}} \cosh\left(\displaystyle\frac{n-2k}{2k}t\right)\right| & \leq & c_{n,k}\varepsilon^{\frac{n+2k}{2k}}e^{\frac{n+2k}{2k}|t|},\\
   \\
   \left|\dot{v}_\varepsilon(t)-\displaystyle\frac{n-2k}{2k}\varepsilon^{\frac{n-2k}{2k}} \sinh\left(\frac{n-2k}{2k}t\right)\right| & \leq & c_{n,k}\varepsilon^{\frac{n+2k}{2k}}e^{\frac{n+2k}{2k}|t|},\\
   \\
   \left|\displaystyle\ddot{v}_\varepsilon(t)-\left(\frac{n-2k}{2k}\right)^2 \varepsilon^{\frac{n-2k}{2k}} \cosh\left(\frac{n-2k}{2k}t\right)\right| & \leq & c_{n,k} \varepsilon^{\frac{n+2k}{2k}}e^{\frac{n+2k}{2k}|t|}.
  \end{array}
$$
\end{proping}
\begin{proof}
Since the Hamiltonian energy $H$ is constant along solutions of (\ref{eq001}) and $v_\varepsilon(0)=\varepsilon^{\frac{n-2k}{2k}}$ is the minimum of $v_\varepsilon$, then $H\left(v_\varepsilon,\dot{v}_\varepsilon\right)=\varepsilon^{n-2k}- \varepsilon^n>0$. From \cite{MR2165306} we have
\begin{equation}\label{eq033}
 h_\varepsilon:=v_\varepsilon^2-\left(\frac{2k}{n-2k}\right)^2\dot v_\varepsilon^2>0.
\end{equation}
Thus
$$\dot{v}_\varepsilon^2=\left(\frac{n-2k}{2k}\right)^2\left(v_\varepsilon^2- \left(v_\varepsilon^{\frac{2kn}{n-2k}}+\varepsilon^{n-2k}-\varepsilon^n \right)^{1/k}\right)$$
and $\varepsilon^{\frac{n-2k}{2k}}\leq v_\varepsilon(t)$, for all $t\in\mathbb{R},$
implies that
\begin{equation*}\label{eq012}
\dot{v}_\varepsilon^2\leq \left(\frac{n-2k}{2k}\right)^2\left( v_\varepsilon^2-\varepsilon^{\frac{n-2k}{k}}\right).
\end{equation*}
Therefore, using that $\cosh t\leq e^{|t|}$ for all $t\in\mathbb R$, we get
\begin{equation}\label{eq002}
v_\varepsilon\leq \varepsilon^{\frac{n-2k}{2k}}\cosh \left(\frac{n-2k}{2k}t\right)\leq \varepsilon^{\frac{n-2k}{2k}} e^{\frac{n-2k}{2k}|t|}.
\end{equation}

Next, writing the equation (\ref{eq001}) for $v_\varepsilon$ as
$$\ddot{v}_\varepsilon- \frac{(n-2k)^2}{4k^2}v_\varepsilon=-\frac{n(n-2k)}{4k^2} v_\varepsilon^{\frac{2kn}{n-2k}-1}\left( v_\varepsilon^2-\left(\frac{2k}{n-2k}\right)^2\dot{v}_\varepsilon^2\right)^{1-k},$$
and noting that $\cosh\left(\frac{n-2k}{2k}t\right)$ satisfies the equation
$$\left(\cosh\left(\frac{n-2k}{2k}t\right)\right)''- \frac{(n-2k)^2}{4k^2}\cosh\left(\frac{n-2k}{2k}t\right)=0,$$
we can represent $v_\varepsilon$ as
\begin{equation}\label{eq011}
 \begin{array}{rcl}
  v_\varepsilon(t) & = & \displaystyle\varepsilon^{\frac{n-2k}{2k}}\cosh\left(\frac{n-2k}{2k}t\right)\\
  \\
  & & - \displaystyle\frac{n(n-2k)}{4k^2}e^{\frac{n-2k}{2k}t}\int_0^t e^{\frac{2k-n}{k}s}\int_0^s e^{\frac{n-2k}{2k}z}v_\varepsilon^{\frac{2kn}{n-2k}-1}(z)\left( v_\varepsilon^2(z)\right.\\
  \\
  & & \displaystyle\left.-\left(\frac{2k}{n-2k}\right)^2\dot{v}_\varepsilon^2(z) \right)^{1-k}dzds.
 \end{array}
\end{equation}

Now, since $H(v_\varepsilon,\dot{v}_\varepsilon)>0$, we get from (\ref{002}) and (\ref{eq002}) that
\begin{equation}\label{eq004}
\begin{array}{rcl}
 v_\varepsilon^{\frac{2kn}{n-2k}-1}\left( v_\varepsilon^2-\left(\frac{2k}{n-2k}\right)^2\dot{v}_\varepsilon^2 \right)^{1-k} & = & \left(\frac{v_\varepsilon^{\frac{2kn}{n-2k}}} {H(v_\varepsilon,\dot{v}_\varepsilon)+v_\varepsilon^{\frac{2kn}{n-2k}}} \right)^{\frac{k-1}{k}} v_\varepsilon^{\frac{n+2k}{n-2k}}\\
 \\
 & \leq & \varepsilon^{\frac{n+2k}{2k}}e^{\frac{n+2k}{2k}|t|}.
\end{array}
\end{equation}

From (\ref{eq002}), (\ref{eq011}) and (\ref{eq004}) we get that for all $t>0$
\begin{equation*}\label{eq003}
   0  \leq  \varepsilon^{\frac{n-2k}{2k}}\cosh\left( \displaystyle\frac{n-2k}{2k}t\right)- v_\varepsilon(t) \leq c_{n,k}\varepsilon^{\frac{n+2k}{2k}} e^{\frac{n+2k}{2k}t},
\end{equation*}
for some constant $c_{n,k}>0$ which depends only on $n$ and $k$

Differentiating the identity (\ref{eq011}), we get
\begin{equation}\label{eq009}
 \dot{v}_\varepsilon(t)=\frac{n-2k}{2k} \varepsilon^{\frac{n-2k}{2k}}\sinh\left(\frac{n-2k}{2k}t \right)-I_1(t)-I_2(t),
\end{equation}
where
$$\begin{array}{rcl}
  I_1(t) & = & \displaystyle \frac{n(n-2k)^2}{(2k)^3}e^{\frac{n-2k}{2k}t}\int_0^t e^{\frac{2k-n}{k}s}\int_0^se^{\frac{n-2k}{2k}z}v_\varepsilon^{\frac{2kn}{n-2k}-1}(z)\left( v_\varepsilon^2(z)\right.\\
  \\
  & & \displaystyle\left.-\left(\frac{2k}{n-2k}\right)^2\dot{v}_\varepsilon^2(z) \right)^{1-k}dzds
  \end{array}
$$
and
$$\begin{array}{rcl}
  I_2(t) & = & \displaystyle\frac{n(n-2k)^2}{(2k)^2}e^{-\frac{n-2k}{2k}t}\int_0^te^{\frac{n-2k}{2k}z}v_\varepsilon^{\frac{2kn}{n-2k}-1}(z)\left( v_\varepsilon^2(z)\right.\\
  \\
  & & \displaystyle\left.-\left(\frac{2k}{n-2k}\right)^2\dot{v}_\varepsilon^2(z) \right)^{1-k}dz.
  \end{array}
$$

Using (\ref{eq002}) and (\ref{eq004}), we get for all $t>0$ that $I_1(t)\leq c_{n,k} \varepsilon^{\frac{n+2k}{2k}} e^{\frac{n+2k}{2k}t}$ and $I_1(t)\leq c_{n,k} \varepsilon^{\frac{n+2k}{2k}} e^{\frac{n+2k}{2k}t}$. From this and (\ref{eq009}) we obtain the second inequality. The third one we obtain in analogous way.
\end{proof}

\begin{proping}\label{propo05}
For any $\varepsilon\in(0,((n-2k)/n)^{1/2k})$ and any $x\in\mathbb{R}^n\backslash\{0\}$ with $|x|\leq 1$, the Delaunay-type solution $u_\varepsilon(x)$ satisfies the estimates
$$\begin{array}{rcl}
  \left|\displaystyle u_\varepsilon(x)- \frac{\varepsilon^{\frac{n-2k}{2k}}}{2}(1+|x|^{\frac{2k-n}{k}})\right| & \leq & c_{n,k}\varepsilon^{\frac{n+2k}{2k}}|x|^{-\frac{n}{k}},\\
  
  \left|\displaystyle |x|\partial_ru_\varepsilon(x)+\frac{n-2k}{2k} \varepsilon^{\frac{n-2k}{2k}}|x|^{\frac{2k-n}{k}}\right| & \leq & c_{n,k}\varepsilon^{\frac{n+2k}{2k}}|x|^{-\frac{n}{k}}\\
  
  \left|\displaystyle |x|^2\partial_r^2u_\varepsilon(x)-\frac{(n-2k)^2}{2k^2} \varepsilon^{\frac{n-2k}{2k}}|x|^{\frac{2k-n}{k}}\right| & \leq & c_{n,k}\varepsilon^{\frac{n+2k}{2k}}|x|^{-\frac{n}{k}},
  \end{array}
$$
for some positive constant $c_{n,k}$ that depends only on $n$ and $k$.
\end{proping}
\begin{proof}
The first inequality follows from the first one in the Proposition \ref{propo003} and noting that for $t=-\log|x|\geq 0$ with $0<|x|<1$ we have $|x|^{\frac{2k-n}{2k}}e^{\frac{n+2k}{2k}|t|}= |x|^{-\frac{n}{k}}$, $u_\varepsilon(x)= |x|^{\frac{2k-n}{2k}}v_\varepsilon(-\log|x|)$ and $\varepsilon^{\frac{n-2k}{2k}}|x|^{\frac{2k-n}{2k}} \cosh\left(\frac{n-2k}{2k}t\right)= \frac{\varepsilon^{\frac{n-2k}{2k}}}{2}\left(1+|x|^{\frac{2k-n}{k}}\right)$.

For the second and third inequality, note that
$$|x|\partial_ru_\varepsilon(x)=\frac{2k-n}{2k}u_\varepsilon(x)-|x|^{\frac{2k-n}{2k}}\dot v_\varepsilon(-\log|x|),$$
and
$$\frac{n-2k}{2k}\varepsilon^{\frac{n-2k}{2k}}|x|^{\frac{2k-n}{2k}}\sinh\left(\frac{n-2k}{2k}t\right)=\frac{n-2k}{2k}\varepsilon^{\frac{n-2k}{2k}}\frac{|x|^{\frac{2k-n}{k}}-1}{2}.$$

Therefore, again by Proposition \ref{propo003} we obtain
$$\left||x|\partial_ru_\varepsilon(x)+\frac{n-2k}{2k} \varepsilon^{\frac{n-2k}{2k}}|x|^{\frac{2k-n}{k}}\right|\leq \left|\frac{2k-n}{2k}\right|\left|u_\varepsilon(x)-\frac{\varepsilon^{\frac{n-2k}{2k}}}{2}(1+|x|^{\frac{2k-n}{k}})\right|$$
$$+\left||x|^{\frac{2k-n}{2k}}\dot v_\varepsilon(-\log|x|)-\frac{n-2k}{2k}\varepsilon^{\frac{n-2k}{2k}}|x|^{\frac{2k-n}{2k}}\sinh\left(\frac{n-2k}{2k}t\right)\right|\leq c_{n,k}\varepsilon^{\frac{n+2k}{2k}}|x|^{-\frac{n}{k}}.$$

In analogous way we get the third inequality.
\end{proof}

For our puposes, it is convenient to consider the following $(n+2)-$dimensional family of solution to (\ref{eq006}) in a small punctured ball centered at the origin
\begin{equation}\label{eq52}
u_{\varepsilon,R,a}(x):= |x-a|x|^2|^{\frac{2k-n}{2k}}v_{\varepsilon} (-2\log|x|+\log|x-a|x|^2|+\log R),
\end{equation}
where only translations along the Delaunay axis and of the ``point at infinity" are allowed (see \cite{MN}). This family of solutions comes from the fact that if $u_\varepsilon$ is a solution then the functions $R^{\frac{2-n}{2}}u_{\varepsilon}(R^{-1}x)$, $u_\varepsilon(x+b)$ and $|x|^{\frac{2k-n}{k}}u_\varepsilon(x|x|^{-2})$ are still solutions in a small punctured ball centered at the origin for any $R>0$ and $b\in\mathbb R^n$. The last function is related with the inversion $I(x)=x|x|^{-2}$ of the $\mathbb R^n\backslash\{0\}$.
\begin{coring}\label{cor02}
There exists a constant $r_0\in(0,1)$, such that for any $x$ and $a$ in $\mathbb{R}^n$ with $|x|\leq 1$, $|a||x|<r_0$, $R\in\mathbb{R}^+$, and  $\varepsilon\in(0,((n-2k)/n)^{1/2k})$ the solution $u_{\varepsilon,R,a}$ satisfies the estimate
\begin{equation}\label{eq37}
u_{\varepsilon,R,a}(x) = u_{\varepsilon,R}(x)+\left(\frac{n-2k}{k}u_{\varepsilon,R}(x)+ |x| \partial_ru_{\varepsilon,R}(x)\right)a\cdot x+O''\left(|a|^2|x|^{\frac{6k-n}{2k}}\right)
\end{equation}
and if $R\leq |x|$, the estimate
\begin{equation}\label{eq93}
\begin{array}{rcl}
 u_{\varepsilon,R,a}(x) & = &\displaystyle u_{\varepsilon,R}(x)+\left(\frac{n-2k}{k}u_{\varepsilon,R}(x)+ |x| \partial_ru_{\varepsilon,R}(x)\right)a\cdot x\\
 \\
 & & \displaystyle+O''\left(|a|^2\varepsilon^{\frac{n-2k}{2k}} R^{\frac{2k-n}{2k}}|x|^2\right).
\end{array}
\end{equation}
\end{coring}
\begin{proof} First note that
\begin{equation}\label{eq035}
 |x-a|x|^2|^{\frac{2k-n}{2k}}=\displaystyle|x|^{\frac{2k-n}{2k}}+\frac{n-2k}{2k}a\cdot x|x|^{\frac{2k-n}{2k}}+O''(|a|^2|x|^{\frac{6k-n}{2k}})
\end{equation}
and
$$\log\left|\displaystyle\frac{x}{|x|}-a|x|\right|=-a\cdot x+O''(|a|^2|x|^2),$$
for $|a||x|<r_0$ and some $r_0\in(0,1)$. Using the Taylor's expansion we obtain that
\begin{equation}\label{eq034}
 \begin{array}{l}
v_\varepsilon\left(-\log |x|+\log\left|\displaystyle \frac{x}{|x|}-a|x|\right|+\log R\right)  =  v_{\varepsilon}(-\log |x|+\log R)\\
\hspace{2cm} -   v_\varepsilon'(-\log|x|+\log R)a\cdot x+  v_\varepsilon'(-\log|x|+\log R)O''(|a|^2|x|^2)\\
\\
\hspace{2cm} +  v_\varepsilon''(-\log|x|+\log R+t_{a,x})O''(|a|^2|x|^2)
\end{array}
\end{equation}
for some $t_{a,x}\in\mathbb{R}$ with $0<|t_{a,x}|<\left|\log\left|\frac{x}{|x|}-a|x|\right|\right|$. Observe that $t_{a,x}\rightarrow 0$ when $|a||x|\rightarrow 0$. From (\ref{eq001}) and (\ref{eq033}) we obtain $|v'_\varepsilon|\leq c_{n,k}v_\varepsilon$,  $|v''_\varepsilon|\leq c_{n,k}v_\varepsilon$. Then, multiplying (\ref{eq035}) by (\ref{eq034}), we get (\ref{eq37}). 

For the second equality, note that if $R\leq |x|$, then $-\log|x|+\log R\leq 0$. Therefore, the result follows by (\ref{eq002}) and
$$|x|\partial_ru_{\varepsilon,R}(x)= \frac{2k-n}{2k}u_{\varepsilon,R}(x)- |x|^{\frac{2k-n}{2k}}v_\varepsilon'(-\log|x|+\log R).$$
\end{proof}
\begin{coring}\label{cor04}
For any $\varepsilon\in(0,((n-2k)/n)^{1/2k})$ and any $x$ in $\mathbb{R}^n$ with $|x|\leq 1$, the function $u_{\varepsilon,R}$ satisfies the estimates
$$ u_{\varepsilon,R}(x) = \displaystyle\frac{\varepsilon^{\frac{n-2k}{2k}}}{2} \left(R^{\frac{2k-n}{2k}}+ R^{\frac{n-2k}{2k}}|x|^{\frac{2k-n}{k}}\right)+ O''(R^{\frac{n+2k}{2k}} \varepsilon^{\frac{n+2k}{2k}}|x|^{-\frac{n}{k}}),$$
$$|x|\partial_ru_{\varepsilon,R}(x)=\frac{2k-n}{2k} \varepsilon^{\frac{n-2k}{2k}} R^{\frac{n-2k}{2k}} |x|^{\frac{2k-n}{k}}+ O'(R^{\frac{n+2k}{2k}} \varepsilon^{\frac{n+2k}{2k}}|x|^{-\frac{n}{k}})$$
and
$$|x|^2\partial_r^2u_{\varepsilon,R}(x)= \frac{(n-2k)^2}{2k^2}\varepsilon^{\frac{n-2k}{2k}} R^{\frac{n-2k}{2k}} |x|^{\frac{2k-n}{k}}+ O(R^{\frac{n+2k}{2k}} \varepsilon^{\frac{n+2k}{2k}}|x|^{-\frac{n}{k}}).$$
\end{coring}
\begin{proof} Directly by the Proposition \ref{propo05}.
\end{proof}
\subsection{Function spaces}\label{sec09}
In this section we define some function spaces that we use in this work. This spaces has appeared in \cite{M}, \cite{MR2194146}, \cite{MR1712628} and \cite{MR2639545}. See these works for more details.
\begin{definition}\label{def1}
For each $k\in\mathbb{N}$, $r>0$, $0<\alpha<1$ and $\sigma\in(0,r/2)$, let $u\in C^k(B_r(0)\backslash\{0\})$, set
$$\|u\|_{(k,\alpha),[\sigma,2\sigma]}=\sup_{|x|\in[\sigma,2\sigma]}\left(\sum_{j=0}^{k}\sigma^j |\nabla^ju(x)|\right)+\sigma^{k+\alpha}\sup_{|x|,|y|\in[\sigma,2\sigma]} \frac{|\nabla^ku(x)-\nabla^ku(y)|}{|x-y|^{\alpha}}.$$
Then, for any $\mu\in\mathbb{R}$, the space $C^{k,\alpha}_{\mu}(B_r(0)\backslash\{0\})$ is the collection of functions $u$ that are locally in $C^{k,\alpha}(B_r(0)\backslash\{0\})$ and for which the norm
$$\|u\|_{(k,\alpha),\mu,r}=\sup_{0<\sigma\leq\frac{r}{2}}\sigma^{-\mu} \|u\|_{(k,\alpha),[\sigma,2\sigma]}$$
is finite.
\end{definition}

Note that $C_{\mu}^{k,\alpha}\subseteq C_{\delta}^{l,\alpha}$ if $\mu\geq\delta$ and $k\geq l$, and $\|u\|_{(l,\alpha),\delta}\leq C\|u\|_{(k,\alpha),\mu}$ for all $u\in C_{\mu}^{k,\alpha}$. 
\begin{definition}\label{def2}
For each $k\in\mathbb{N}$, $0<\alpha<1$ and $r>0$. The space $C^{k,\alpha}(\mathbb{S}^{n-1}_r)$ is the collection of functions $\phi\in C^k(\mathbb{S}^{n-1}_r)$ for which the norm 
$$\|\phi\|_{(k,\alpha),r}:=\|\phi(r\cdot)\|_{C^{k,\alpha}(\mathbb{S}^{n-1})}.$$
is finite.
\end{definition}

We often will write 
$$C^{k,\alpha}(\mathbb{S}_r^{n-1})^{\perp}:=\left\{\phi\in C^{k,\alpha}(\mathbb{S}_r^{n-1});\;\pi''_r(\phi)=\phi\right\},$$
$$C^{k,\alpha}_\mu(B_r(0)\backslash\{0\})^{\perp}:=\left\{u\in C^{k,\alpha}_\mu(B_r(0) \backslash\{0\});\;\pi''_s(u(s\cdot)) = u(s\cdot), \forall s\in(0,r)\right\}.$$
and
$$C_{{\mu}}^{k,\alpha}(B_r(0)\backslash\{0\})^\top :=\left\{u\in C_{{\mu}}^{k,\alpha}(B_r(0)\backslash\{0\}); \pi_s''(u(s\cdot))=0, \forall s\in(0,r)\right\}.$$

Next, consider $(M,g)$ an $n-$dimensional compact Riemannian manifold and $\Psi:B_{r_1}(0)\rightarrow M$ some coordinate system on $M$ centered at some point $p\in M$, where $B_{r_1}(0)\subset\mathbb{R}^n$ is the ball of radius $r_1>0$ centered in the origin. For $0<r<s\leq r_1$ define $M_r:=M\backslash\Psi(B_r(0))$ and $\Omega_{r,s}:=\Psi(A_{r,s}),$ where $A_{r,s}:=\{x\in\mathbb{R}^n;r\leq|x|\leq s\}$.
\begin{definition}\label{def4}
For all $k\in\mathbb{N}$, $\alpha\in(0,1)$, $0<r<s\leq r_1$ and $\mu\in\mathbb{R}$, the spaces $C_\mu^{k,\alpha}(\Omega_{r,s})$ and $C_\mu^{k,\alpha}(M_r)$ are the spaces of functions $v\in C_{loc}^{k,\alpha}(M\backslash\{p\})$ for which the following norms 
$$\|v\|_{C^{k,\alpha}_\mu(\Omega_{r,s})}:=\sup_{r\leq \sigma\leq\frac{s}{2}}\sigma^{-\mu}\|v\circ \Psi\|_{(k,\alpha),[\sigma,2\sigma]}$$
and
$$\|v\|_{C^{k,\alpha}_\mu(M_r)}:=\|v\|_{C^{k,\alpha}(M_{\frac{1}{2}r_1})}+ \|v\|_{C^{k,\alpha}_\mu(\Omega_{r,r_1})},$$
respectively, are finite.
\end{definition}
\subsection{The linearized operator about the Delaunay-type solutions}\label{sec01}

Since we will need of the inverse of the linearized operator, we start this section recalling the expression for the linearized operator about the Delaunay-type solution $v_\varepsilon$ and a proposition from \cite{MN} which gives a right inverse for it.
\begin{lemma}[Mazzieri-Ndiaye \cite{MN}]\label{lem000}
In the cylinder $\mathbb S^{n-1}\times \mathbb R$ the linearization of the operator defined in (\ref{eq057}) about the Delaunay-type solution $v_\varepsilon$ is given by
$$\begin{array}{ccl}
   L_{g_{cyl}}^{v_\varepsilon}(w) & = & C_{n,k}v_\varepsilon h_\varepsilon^{k-1}\left(\dfrac{\partial^2}{\partial t^2}+a_\varepsilon\Delta_\theta+b_\varepsilon\dfrac{\partial}{\partial t}+c_\varepsilon\right)w\\
   \\
   & = &C_{n,k}v_\varepsilon h_\varepsilon^{\frac{k-1}{2}}\mathcal L_\varepsilon(h_\varepsilon^{\frac{k-1}{2}}w),
  \end{array}
$$
where $\Delta_\theta$ is the Laplace-Beltrami operator for standard round metric on the unit sphere,
$$\begin{array}{ccl}
   \mathcal L_\varepsilon & = & \dfrac{\partial^2}{\partial t^2}+a_\varepsilon\Delta_\theta+c_\varepsilon+d_\varepsilon, 
   \\
  a_\varepsilon & = & 1-\dfrac{n(k-1)}{k(n-1)}\dfrac{H_\varepsilon}{h_\varepsilon^k}=\dfrac{n-k}{k(n-1)}+\dfrac{v_\varepsilon^{\frac{2kn}{n-2k}}}{h_\varepsilon^k},\\
  \\
  b_\varepsilon & = & -\left(\dfrac{2(n-k)}{n-2k}-\dfrac{2k(n-1)}{n-2k}a_\varepsilon\right)\dfrac{\dot v_\varepsilon}{v_\varepsilon}=(k-1)\dfrac{\dot h_\varepsilon}{h_\varepsilon},\\
  \\
  c_\varepsilon & = & -\dfrac{(n-1)(n-2k)}{2k}a_\varepsilon+\dfrac{n-2k}{2k}+\dfrac{\ddot{v}_\varepsilon}{v_\varepsilon}+\dfrac{n^2}{2k}h^{1-k}v_\varepsilon^{\frac{2kn}{n-2k}-2},\\
  \\
  d_\varepsilon & = & -\dfrac{k-1}{2}\dfrac{\partial^2}{\partial t^2}\log h_\varepsilon-\left(\dfrac{k-1}{2}\right)^2\left(\dfrac{\partial }{\partial t}\log h_\varepsilon\right)^2,
  \end{array}
$$
$H_\varepsilon:=H(v_\varepsilon,\dot v_\varepsilon)$ is the Hamiltonian energy, $h_\varepsilon$ is defined in (\ref{eq033}) and the constant $C_{n,k}$ depends only on $n$ and $k$.
\end{lemma}

For the next proposition we set
\begin{equation*}\label{eq038}
 \delta_{n,k}:=\sqrt{\frac{2n(n-k)}{k(n-1)}+ \left(\frac{n-2k}{2k}\right)^2}.
\end{equation*}
Note that $\delta_{n,k}+1-n/2k<2$ for all $k>1$. Define $D_1=(0,+\infty)\times\mathbb S^{n}$.

\begin{proping}[Mazzieri-Ndiaye \cite{MN}]\label{propo001}
 Let $R>0$, $\gamma\in\left(-\delta_{n,k},\delta_{n,k}\right)$ and $\overline\gamma>n/2$. There exists a positive real number $\varepsilon_0=\varepsilon_0(\gamma,\overline \gamma,n,k,\beta)>0$ such that for every $\varepsilon\in(0,\varepsilon_0]$, the bounded linear operator
 $$\mathcal L_{\varepsilon,R}:[C_\gamma^{2,\beta}(D_1)^\perp\oplus C_{\overline \gamma}^{2,\beta}(D_1)^\top\oplus \mathcal W_\varepsilon(D_1)]_0\rightarrow C_\gamma^{0,\beta}(D_1)^\perp\oplus C_{\overline\gamma}^{0,\beta}(D_1)^\top$$
 is an isomorphism, where $\mathcal W_\varepsilon(D_1)$ is a finite vetorial space called the deficiency space, which is generated by the Jacobi fields. Moreover, if $w\in C_\gamma^{2,\beta}(D_1)^\perp\oplus C_{\overline \gamma}^{2,\beta}(D_1)^\top\oplus \mathcal W_\varepsilon(D_1)$ and $f\in C_\gamma^{0,\beta}(D_1)^\perp\oplus C_{\overline\gamma}^{0,\beta}(D_1)^\top$ verify $\mathcal L_{\varepsilon,R}w=f$, and, with the notations introduced above, we decompose $w$ and $f$ as
 $$w=w^\perp+w^\top+h_{\varepsilon}^{\frac{k-1}{2}}\sum_{j=0}^na_j\Psi_\varepsilon^{j,+}\;\;\;\mbox{and}\;\;\;f=f^\perp+f^\top,$$
 then we have that there exists a positive constant $C=C(\gamma,\overline\gamma,n,k,\beta)>0$ such that, for every $\varepsilon\in(0,\varepsilon_0]$,
 $$\|w^\perp\|_{C^{2,\beta}_\gamma(D_1)}\leq C\|f^\perp\|_{C^{0,\beta}_\gamma(D_1)},$$
 $$\|w^\top\|_{C^{2,\beta}_{\overline\gamma}(D_1)}\leq C\|f^\top\|_{C^{0,\beta}_{\hat\gamma}(D_1)}$$
 and 
 $$\varepsilon^{\frac{n-2k}{2}}\sum_{j=0}^n|a_j|\leq C\|\hat f\|_{C^{0,\beta}_{\hat\gamma}(D_1)}.$$
\begin{proof}
 A carefully reading of the proof in \cite{MN} we see that the constant $C$ does not depend on $R$.
\end{proof}
\end{proping}

From this proposition we get the following
\begin{proping}
Let $R>0$, $\gamma\in\left(-\delta_{n,k}+1-n/2k,\delta_{n,k}+1-n/2k\right)$ and $\overline\gamma>\displaystyle n/2+1-n/2k$. There exists a positive real number $\varepsilon_0=\varepsilon_0(\gamma,\overline\gamma, n, k,\beta)>0$ such that, for every $\varepsilon\in(0,\varepsilon_0]$, there is an operator
$$\hspace{-2cm}G_{\varepsilon,R}:C^{0,\alpha}_{\gamma-1-n+\frac{n}{2k}} (B_1(0)\backslash\{0\})^\perp\oplus C^{0,\alpha}_{\overline\gamma-1-n+\frac{n}{2k}} (B_1(0)\backslash\{0\})^\top$$
$$\hspace{4cm}\longrightarrow C^{2,\alpha}_{\gamma}(B_1(0)\backslash\{0\})^\perp\oplus C^{2,\alpha}_{\overline\gamma}(B_1(0)\backslash\{0\})^\top$$
such that for $f=f^\perp+f^\top\in C^{0,\alpha}_{\gamma} (B_1(0)\backslash\{0\})^\perp\oplus C^{0,\alpha}_{\overline\gamma} (B_1(0)\backslash\{0\})^\top$, the function $w:=G_{\varepsilon,R}(f)=w^\perp+w^\top$ solves the equation
$$\left\{\begin{array}{rcl}
L_{\delta}^{u_{\varepsilon,R}}(w)=f & \mbox{ in } & B_1(0)\backslash\{0\}\\
\pi''_1(w|_{\mathbb{S}^{n-1}})=0 & on & \partial B_1(0)
\end{array}\right.,$$
and the norm satisfies
 \begin{equation*}\label{eq028}
  \|w^\perp\|_{C^{2,\beta}_\gamma(B_1(0)\backslash\{0\})}\leq C\|f^\perp\|_{C^{0,\beta}_{\gamma-1-n+\frac{n}{2k}}(B_1(0)\backslash\{0\})}
 \end{equation*}
 and
 $$\|w^\top\|_{C^{2,\beta}_{\overline\gamma}(B_1(0)\backslash\{0\})}\leq C\|f^\top\|_{C^{0,\beta}_{\overline\gamma-1-n+\frac{n}{2k}}(B_1(0)\backslash\{0\})},$$
 where $C=C(\gamma,\overline\gamma,n,k,\beta)>0$ is a constant.	

\end{proping}
\begin{proof}
 Since $\Phi^*\delta= (e^{\frac{2k-n}{2k}t})^{\frac{4k}{n-2k}}g_{cyl}$, then using the conformal equivariance (\ref{eq024}) we obtain
\begin{equation}\label{eq023}
L_\delta^{u_{\varepsilon,R}}(w)\circ\Phi= e^{nt}L_{g_{cyl}}^{v_{\varepsilon,R}} (e^{\frac{2k-n}{2k}t}w\circ\Phi).
\end{equation} 
From this and Proposition \ref{propo001} the result follows.
\end{proof}

Let $f$ defined in $B_1(0)\backslash\{0\}$ and $v$ solution of
$$L_{\varepsilon,R}(v)=f \mbox{ in } B_1(0)\backslash\{0\}.$$
Here we are doing $L_{\varepsilon,R}:=L_\delta^{u_{\varepsilon,R}}$. 

Note that, for $r>0$, if $v_r(x)=v(r^{-1}x)$, then by Proposition \ref{propo001} and (\ref{eq023}) we get
$$L_\delta^{u_{\varepsilon,rR}}(v_r)(x)=r^{-1-n+\frac{n}{2k}}L_\delta^{u_{\varepsilon,R}}(v)(r^{-1}x).$$
So, if we define $g(x)=r^{-1-n+\frac{n}{2k}}f(r^{-1}x)$, then
$$L_{\varepsilon,rR}(v_r)=g \mbox{ in } B_r(0)\backslash\{0\}.$$

Besides, the norm satisfies $\|v_r\|_{(2,\alpha),\mu,r}=r^{-\mu}\|v\|_{(2,\alpha),\mu,1}$ and $\|g\|_{(2,\alpha),\mu,r}=r^{-\mu}\|f\|_{(2,\alpha),\mu,1}.$ Therefore we obtain the next result.

\begin{proping}\label{propo004}
Let $R>0$, $\gamma\in\left(-\delta_{n,k}+1-n/2k,\delta_{n,k}+1-n/2k\right)$ and $\overline\gamma>\displaystyle n/2+1-n/2k$. There exists a positive real number $\varepsilon_0=\varepsilon_0(\gamma,\overline\gamma, n, k,\beta)>0$ such that, for every $\varepsilon\in(0,\varepsilon_0]$, there is an operator
$$G_{\varepsilon,R,r}:C^{0,\alpha}_{\gamma-1-n+\frac{n}{2k}} (B_r(0)\backslash\{0\})^\perp\oplus C^{0,\alpha}_{\overline\gamma-1-n+\frac{n}{2k}} (B_r(0)\backslash\{0\})^\top$$
$$\hspace{3cm}\longrightarrow C^{2,\alpha}_{\gamma}(B_r(0)\backslash\{0\})^\perp\oplus C^{2,\alpha}_{\overline\gamma}(B_r(0)\backslash\{0\})^\top$$
such that for $f=f^\perp+f^\top\in C^{0,\alpha}_{\gamma} (B_r(0)\backslash\{0\})^\perp\oplus C^{0,\alpha}_{\overline\gamma} (B_r(0)\backslash\{0\})^\top$, the function $w:=G_{\varepsilon,R}(f)=w^\perp+w^\top$ solves the equation
$$\left\{\begin{array}{rcl}
L_{\delta}^{u_{\varepsilon,R}}(w)=f & \mbox{ in } & B_r(0)\backslash\{0\}\\
\pi''_r(w|_{\mathbb{S}^{n-1}})=0 & on & \partial B_r(0)
\end{array}\right.,$$
and the norm satisfies
 $$\|w^\perp\|_{C^{2,\beta}_\gamma(B_r(0)\backslash\{0\})}\leq C\|f^\perp\|_{C^{0,\beta}_{\gamma-1-n+\frac{n}{2k}}(B_r(0)\backslash\{0\})}$$
 and
 $$\|w^\top\|_{C^{2,\beta}_{\overline\gamma}(B_r(0)\backslash\{0\})}\leq C\|f^\top\|_{C^{0,\beta}_{\overline\gamma-1-n+\frac{n}{2k}}(B_r(0)\backslash\{0\})},$$
 where $C=C(\gamma,\overline\gamma,n,k,\beta)>0$ is a constant.	

\end{proping}

For the next result we define the space
$$C_{(\mu,\nu)}^{k,\alpha}\left(B_r(0)\backslash\{0\} \right):= C_{\mu}^{k,\alpha}(B_r(0)\backslash\{0\})^\perp\oplus C_{\nu}^{k,\alpha}(B_r(0)\backslash\{0\})^\top,$$
with the norm 
\begin{equation}\label{eq058}
 \|u\|_{(k,\alpha),(\mu,\nu)} := r^{\gamma-\overline\gamma}\|u^\perp\|_{(k,\alpha),\mu,}+ \| u^\top\|_{(k,\alpha),\nu},
\end{equation}
where $\gamma$ and $\overline\gamma$ will be given by the proposition and $u=u^\perp+u^\top$ with $u^\perp\in C_\mu^{m,\alpha}(B_r(0)\backslash\{0\})^\perp$ and $u^\top\in C_{\nu}^{m,\alpha}(B_r(0)\backslash\{0\})^\top.$

By a perturbation argument we obtain the next corollary. Here we define $L_{\varepsilon,R,a}:=L_\delta^{u_{\varepsilon,R,a}}.$
\begin{coring}\label{cor001}
Let $R>0$, $\alpha\in(0,1)$, $\gamma\in\left(-\delta_{n,k}+1-n/2k, \delta_{n,k}+1-n/2k\right)$ and $\overline\gamma>n/2+1-n/2k$. There exist positive real numbers $\varepsilon_0=\varepsilon_0(\mu,\delta, n, \alpha)>0$ and $r_0>0$, such that, for every $\varepsilon\in(0,\varepsilon_0]$, $a\in\mathbb{R}^n$ and $r\in(0,1]$ with $|a|r\leq r_0$, there is an operator
$$G_{\varepsilon,R,r,a}:C^{0,\alpha}_{(\gamma-1-n+\frac{n}{2k},\overline\gamma-1-n+\frac{n}{2k})} (B_r(0)\backslash\{0\})\rightarrow C^{2,\alpha}_{(\gamma,\overline\gamma)}(B_r(0)\backslash\{0\})$$
with the norm  bounded independently of $\varepsilon$ and $R$, such that for every function $f\in C^{0,\alpha}_{(\gamma-1-n+\frac{n}{2k},\overline\gamma-1-n+\frac{n}{2k})} (B_r(0)\backslash\{0\})$, the function $w:=G_{\varepsilon,R,r,a}(f)$ solves the equation
$$\left\{\begin{array}{rcl}
L_{\varepsilon,R,a}(w)=f & \mbox{ in } & B_r(0)\backslash\{0\}\\
\pi''_r(w|_{\mathbb{S}^{n-1}})=0 & on & \partial B_r(0)
\end{array}\right.$$
and the norm satisfies
 \begin{equation}\label{eq047}
  \|w\|_{(2,\alpha),(\gamma,\overline\gamma)}\leq C\|f\|_{(0,\beta),(\gamma-1-n+\frac{n}{2k},\overline\gamma-1-n+\frac{n}{2k})},
 \end{equation}
 where $C=C(\gamma,\overline\gamma,n,k,\alpha)>0$ is a constant.	
\end{coring}
\begin{proof} By Lemma \ref{lem000} and (\ref{eq023}) we obtain
$$\|L^{u_{\varepsilon,R,a}}_\delta(v)-L^{u_{\varepsilon,R}}_\delta(v)\|_{(0,\alpha),[\sigma,2\sigma]}\leq c_{n,k}|a|\sigma^{-n+\frac{n}{2k}}\|v\|_{(2,\alpha),[\sigma,2\sigma]}.$$
This inequality holds for the low and high frequencies spaces. Thus by definition of the norm (\ref{eq058}) we get that
\begin{equation*}\label{eq026}
 \|L^{u_{\varepsilon,R,a}}_\delta(v)-L^{u_{\varepsilon,R}}_\delta(v)\|_{(0,\alpha),(\gamma-1-n+\frac{n}{2k},\overline\gamma-1-n+\frac{n}{2k})}\leq c_{n,k}|a|r\|v\|_{(2,\alpha),(\gamma,\overline\gamma)}.
\end{equation*}

 Therefore if we choose $r_0\leq (2c_{n,k}\|G_{\varepsilon,R,r}\|)^{-1}$, where $G_{\varepsilon,R,r}$ is the operator given by Proposition \ref{propo004}, then
 $$\|L_{\varepsilon,R,a}\circ G_{\varepsilon,R,r}-I\|\leq\|L_{\varepsilon,R,a}-L_{\varepsilon,R}\|\|G_{\varepsilon,R,r}\|\leq \frac{1}{2}.$$
 
 This implies that the operator $L_{\varepsilon,R,a}\circ G_{\varepsilon,R,r}$ has a bounded right inverse given by
 $$\left(L_{\varepsilon,R,a}\circ G_{\varepsilon,R,a}\right)^{-1}=\sum_{i=0}^{\infty}\left(I-L_{\varepsilon,R,a}\circ G_{\varepsilon,R,r}\right)^i.$$
 and it has norm bounded independently of $\varepsilon$, $R$, $a$ and $r$,
 $$\|\left(\pi'\circ L_{\varepsilon,R,a}\circ G_{\varepsilon,R,a}\right)^{-1}\|\leq \sum_{i=0}^\infty\|L_{\varepsilon,R,a}\circ G_{\varepsilon,R,a}-I\|^i\leq 1.$$
 
 Therefore we define a right inverse for $L_{\varepsilon,R,a}$ by 
 $$G_{\varepsilon,R,r,a}:=G_{\varepsilon,R,r}\circ \left(L_{\varepsilon,R,a}\circ G_{\varepsilon,R,r}\right)^{-1}.$$
 \end{proof}

\section{Interior Analysis}\label{sec02}
In this section we will explain how to use the assumption on the Weyl tensor to reduce the problem to a problem of finding a fixed point of a map. We will show the existence of a family of local solutions, for the singular $\sigma_2-$Yamabe problem, in some punctured small ball centered at a point $p$, which depends on $n+2$ parameters with prescribed Dirichlet data. Moreover, each element of this family is asymptotic to a Delaunay-type solution $u_{\varepsilon,R,a}$.

\subsection{ Nonlinear analysis}
Throughout the rest of the paper $d=\left[\frac{n}{2}\right]$. Recall that $(M,g_0)$ is a compact Riemannian manifold with dimension $n\geq 5$, $\sigma_2(A_{g_0})=n(n-1)/8$, and the Weyl tensor $W_{g_0}$ at $p\in M$ satisfies the condition
\begin{equation}\label{eq059}
 \nabla^lW_{g_0}(p)=0,\;\;l=0,1,\ldots,d-2.
\end{equation}

Since our problem is conformally invariant, in this section we will work in conformal normal coordinates given by Theorem 2.7 in \cite{MR888880}. By its proof there exists a positive smooth function $\mathcal F\in C^\infty(M)$ such that $g=\mathcal F^{\frac{8}{n-4}}g_0$ and $\mathcal F(x)=1+\overline f$, with $\overline{f}=O(|x|^2)$ in $g-$normal coordinates at $p$. Also, since the Weyl tensor is conformally invariant, it follows that the Weyl tensor of the metric $g$ satisfies the condition (\ref{eq059}). 

In theses coordinates it is convenient to consider the Taylor expansion of the metric.  We will write $g_{ij}=\exp(h_{ij}),$ where $h_{ij}$ is a symmetric two-tensor satisfying $h_{ij}=O(|x|^2)$ and tr$h_{ij}(x)=O(|x|^N)$, where $N$ is as big as we want. In this case $\det(g_{ij})=1+O(|x|^N)$. Using the assumption of the Weyl tensor (\ref{eq059}), we obtain $h_{ij}=O(|x|^{d+1})$. Therefore, we conclude that $g=\delta_{ij}+O(|x|^{d+1})$, $R_g=O(|x|^{d-1})$ and $|Ric_g|=O(|x|^{d-1})$.

Next lets recall from \cite{MR2639545} the following proposition, see also \cite{M}.

\begin{proping}\label{propo02}
Let $\mu\leq 2$, $0<r<1$ and $\alpha\in(0,1)$ be constants. For each $\phi\in C^{2,\alpha}(\mathbb{S}_r^{n-1})^\perp$ there is a function $v_\phi\in C^{2,\alpha}_2(B_r(0)\backslash\{0\})^\perp$ so that
$$\left\{\begin{array}{lcl}
\Delta v_\phi=0 & \mbox{ in } & B_r(0)\backslash\{0\}\\
v_\phi=\phi & \mbox{ on } & \partial B_r(0)
\end{array}\right.$$
and
\begin{equation}\label{eq73}
\|v_\phi\|_{(2,\alpha),\mu,r}\leq C r^{-\mu}\|\phi\|_{(2,\alpha),r},
\end{equation}
where the constant $C>0$ does not depend on $r$
\end{proping}

The main goal of this section is to solve the PDE
\begin{equation}\label{eq05}
H_{g}(u_{\varepsilon,R,a}+r^{-\overline\gamma}|x|^{\overline\gamma}v_\phi+h+v)=0
\end{equation}
in $B_{r}(0)\backslash\{0\}\subset\mathbb{R}^n$ for some $r>0$, $\varepsilon>0$, $R>0$, $\phi\in C^{2,\alpha}(\mathbb{S}_r^{n-1})^\perp$, $a\in\mathbb{R}^n$ and $\overline\gamma>1+n/4$, with $u_{\varepsilon,R,a}+r^{-\overline\gamma}|x|^{\overline\gamma}v_\phi+h+v>0$ and prescribed Dirichlet data, where the operator $H_{g}$ is defined in (\ref{eq014}) and $u_{\varepsilon,R,a}$ in (\ref{eq52}). Here, the function $h$ is defined as
$$h=\frac{1}{2}\left((1-\overline\gamma)r^{-\overline\gamma-1}|x|^{\overline\gamma+1}+(\gamma+1)r^{-\overline\gamma+1}|x|^{\overline\gamma-1}\right)  f,$$
where $f=O(|x|^2)$ will be chosen later. Thus, we have $h=O(|x|^{\overline \gamma+1})$. In the Section \ref{sec05} we will explain why to consider this function.

Note that since $H_\delta(u_{\varepsilon,R,a})=0$ by the Taylor's expansion we see that (\ref{eq05}) is equivalent to
\begin{equation}\label{eq10}
\begin{array}{lcl}
L_{\varepsilon,R,a}(v) & = & L_{\varepsilon,R,a}(v)- L_{g}^{u_{\varepsilon,R,a}}(v)- Q_{\varepsilon,R,a}(r^{-\overline\gamma}|x|^{\overline\gamma}v_\phi+h+v) \\
\\
& + & H_\delta(u_{\varepsilon,R,a})-H_g(u_{\varepsilon,R,a})- L_g^{u_{\varepsilon,R,a}}( r^{-\overline\gamma}|x|^{\overline\gamma}v_\phi+h),
\end{array}
\end{equation}
where $Q_{\varepsilon,R,a}(v):=Q^{u_{\varepsilon,R,a}}(v)$ is defined in (\ref{eq039}).

Therefore, by using the right inverse given by Corollary \ref{cor001}, we can find a solution to this equation by finding a fixed point
\begin{equation}\label{eq027}
\begin{array}{lcl}
v & = & G_{\varepsilon,R,a,r}\left(L_{\varepsilon,R,a}(v)- L_{g}^{u_{\varepsilon,R,a}}(v)- Q_{\varepsilon,R,a}(r^{-\overline\gamma}|x|^{\overline\gamma}v_\phi+h+v) \right.\\
\\
& + &\left. H_\delta(u_{\varepsilon,R,a})-H_g(u_{\varepsilon,R,a})- L_g^{u_{\varepsilon,R,a}}( r^{-\overline\gamma}|x|^{\overline\gamma}v_\phi+h)\right).
\end{array}
\end{equation}

But, first we have to show that the right hand side of (\ref{eq027}) is well defined, that is, all terms of the right hand side of the equation (\ref{eq10}) belongs to the right space, which is $C^{0,\alpha}_{(\gamma-1-\frac{3n}{4},\overline\gamma-1-\frac{3n}{4})} (B_r(0)\backslash\{0\})$, for some $\gamma\in(-\delta_{n,2}+1-n/4,\delta_{n,2}+1-n/4)$ and $\overline{\gamma}>1+n/4$.
\begin{lemma}\label{lem05}
Let $\gamma=\delta_{n,2}+1-n/4-\varepsilon_1$ and  $\overline{\gamma}=n/4+1+\varepsilon_1$, where $\varepsilon_1>0$ is very small. For all $v\in C^{2,\alpha}_{(\gamma,\overline\gamma)} (B_r(0)\backslash\{0\})$ and $\phi\in C^{2,\alpha}(\mathbb{S}_{r}^{n-1})^\perp$,  the right hand side of (\ref{eq10}) belongs to $C^{0,\alpha}_{(\gamma-1-\frac{3n}{4},\overline\gamma-1-\frac{3n}{4})} (B_r(0)\backslash\{0\})$.
\end{lemma}
\begin{proof} If $v\in C^{2,\alpha}_{(\gamma,\overline\gamma)} (B_r(0)\backslash\{0\})$, then $v=v^\perp+v^\top$ with $v^\perp=O(|x|^{\gamma})$ and $v^\top=O(|x|^{\overline\gamma})$. Since $\gamma<\overline\gamma$, we have $v=O(|x|^\gamma)$ and $r^{-\overline\gamma}|x|^{\overline\gamma}v_\phi+h+v=O(|x|^\gamma)$. Thus, using (\ref{eq014}) we obtain
\begin{equation}\label{eq040}
 H_\delta(u_{\varepsilon,R,a})-H_g(u_{\varepsilon,R,a}) = O(|x|^{d+1-n})= O(|x|^{\overline\gamma-1-\frac{3n}{4}}),
\end{equation}
since $d\geq\overline\gamma+n/4-2$. Now, by (\ref{eq037}) we get
$$\nonumber L_{\varepsilon,R,a}(v^\perp)- L_{g}^{u_{\varepsilon,R,a}}(v^\perp) = O(|x|^{d+\gamma-\frac{3n}{4}})=O(|x|^{\overline\gamma-1-\frac{3n}{4}}),$$
since $d+\gamma\geq\overline\gamma-1$, and also
$$L_{\varepsilon,R,a}(v^\top)- L_{g}^{u_{\varepsilon,R,a}}(v^\top) = O(|x|^{d+\overline\gamma-\frac{3n}{4}})=O(|x|^{\overline\gamma-1-\frac{3n}{4}}).$$
Using the definition (\ref{eq037}) and (\ref{eq039}) we obtain
\begin{equation*}\label{eq043}
 Q_{\varepsilon,R,a}(r^{-\overline\gamma}|x|^{\overline\gamma}v_\phi+h+v) = O(|x|^{2\gamma-2-\frac{n}{2}})=O(|x|^{\overline\gamma-1-\frac{3n}{4}}),
\end{equation*}
since $2\gamma-1+n/4\geq \overline\gamma$. Finally, using again (\ref{eq037}) and the fact that $r^{-\overline\gamma}|x|^{\overline\gamma}v_\phi+h=O(|x|^{\overline\gamma+1})$, we obtain that
$$L_g^{u_{\varepsilon,R,a}}( r^{-\overline\gamma}|x|^{\overline\gamma}v_\phi+h) = O(|x|^{\overline\gamma-\frac{3n}{4}}) = O(|x|^{\overline\gamma-1-\frac{3n}{4}}).$$

From theses estimates and the definition of the norm (\ref{eq058}), we obtain the result.
\end{proof}
Let $\gamma=\delta_{n,2}+1-n/4-\varepsilon_1$ and  $\overline{\gamma}=n/4+1+\varepsilon_1$, where $\varepsilon_1>0$ is small. To solve the equation (\ref{eq05}) we need to show that the map $N_\varepsilon(R,a,\phi,\cdot): C^{2,\alpha}_{(\gamma,\overline\gamma)} (B_r(0)\backslash\{0\})\rightarrow C^{2,\alpha}_{(\gamma,\overline\gamma)} (B_r(0)\backslash\{0\})$ has a fixed point for suitable parameters $\varepsilon$, $R$, $a$ and $\phi$. Here $N_\varepsilon(R,a,\phi,\cdot)$ is defined by
\begin{equation}\label{eq016}
\begin{array}{rcl}
N_\varepsilon(R,a,\phi,v) & = & G_{\varepsilon,R,r,a}\left( L_{\varepsilon,R,a}(v)- L_{g}^{u_{\varepsilon,R,a}}(v)\right.\\
\\
& - & Q_{\varepsilon,R,a}(r^{-\overline{\gamma}}|x|^{\overline\gamma}v_\phi+h+v) \\
\\
& + & H_\delta(u_{\varepsilon,R,a})-H_g(u_{\varepsilon,R,a})\\
\\
& - & \left.L_g^{u_{\varepsilon,R,a}}(r^{-\overline{\gamma}}|x|^{\overline\gamma}v_\phi+h)\right),
\end{array}
\end{equation}
where $G_{\varepsilon,R,r,a}$ is a right inverse for $L_{\varepsilon,R,a}$ given by Corollary \ref{cor001}.

\subsection{Complete Delaunay-type ends with constant $\sigma_2-$curvature}\label{sec12}

In this section we will show that the map (\ref{eq016}) has a fixed point. Next  we will prove the main result of this section. 

\begin{remark}
 To ensure some estimates that we will need, from now on, we will consider 
 $$R^{\frac{4-n}{4}}=2(1+b)\varepsilon^{\frac{4-n}{4}},$$
 with $|b|\leq 1/2$. Also, we will consider $r_\varepsilon=\varepsilon^s$, for some $0<s<1$.
\end{remark}

\begin{proping}\label{propo03}
Let $\gamma=\delta_{n,2}+1-n/4-\varepsilon_1$ and $\overline{\gamma}=n/4+1+\varepsilon_1$ be constants, where $\varepsilon_1>0$ is a small constant. There exists a constant $\varepsilon_0\in(0,1)$ such that for each $\varepsilon\in(0,\varepsilon_0)$, $\kappa>0$, $\tau>0$, $|b|<1/2$, $a\in\mathbb{R}^n$,  $\delta_1,\delta_2, l\in\mathbb R_+$ small, and $\phi\in C^{2,\alpha}(\mathbb{S}_{r_\varepsilon}^{n-1})^\perp$ with $3\delta_2>\max\{\delta_1,l\}$, $|a|r_\varepsilon^{1-\delta_2}\leq 1$ and $\|\phi\|_{(2,\alpha),r}\leq \kappa r_\varepsilon^{2+l-\delta_1}$, there exists a fixed point $u\in C^{2,\alpha}_{(\gamma,\overline\gamma)} (B_{r_\varepsilon}(0)\backslash\{0\})$ of the map $N_\varepsilon(R,a,\phi,\cdot)$ in the ball of radius $\tau r_\varepsilon^{2+l-\overline\gamma}$.
\end{proping}
\begin{proof} First remember that we are using the norm
$$\|v\|_{(2,\alpha),(\gamma,\overline\gamma)}={r_\varepsilon}^{\gamma-\overline\gamma}\|v^\perp\|_{(2,\alpha),\gamma}+\|v^\top\|_{(2,\alpha),\overline\gamma}.$$
Then, since $r_\varepsilon<1$ and $\gamma<\overline\gamma$, we obtain that
\begin{equation*}\label{eq048}
 \|v\|_{(2,\alpha),[\sigma,2\sigma]}\leq \sigma^\gamma\|v^\perp\|_{(2,\alpha),\gamma}+\sigma^{\overline\gamma}\|v^\top\|_{(2,\alpha),\overline\gamma}\leq \sigma^\gamma\|v\|_{(2,\alpha),(\gamma,\overline\gamma)}.
\end{equation*}

Note that
\begin{equation*}
\begin{array}{rcl}
N_\varepsilon(R,a,\phi,0) & = & G_{\varepsilon,R,r,a}\left( - Q_{\varepsilon,R,a}({r_\varepsilon}^{-\overline\gamma}|x|^{\overline\gamma}v_\phi+h) \right.\\
\\
& + & H_\delta(u_{\varepsilon,R,a})-H_g(u_{\varepsilon,R,a})\\
\\
& - &\left. L_g^{u_{\varepsilon,R,a}}({r_\varepsilon}^{-\overline\gamma}|x|^{\overline\gamma}v_\phi+h)\right)
\end{array}
\end{equation*}
and
$$N_\varepsilon(R,a,\phi,v_1)- N_\varepsilon(R,a,\phi,v_2)= G_{\varepsilon,R,r,a}\left(L_{\varepsilon,R,a}(v_1-v_2)- L_{g}^{u_{\varepsilon,R,a}}(v_1-v_2)\right.$$
$$\left.-\displaystyle\int_0^1\frac{d}{dt} Q_{\varepsilon,R,a}({r_\varepsilon}^{-\overline\gamma}|x|^{\overline\gamma}v_\phi+h+v_2+t(v_1-v_2))dt\right).$$

By (\ref{eq040}) we get
$$\sigma^{-\mu+1+\frac{3n}{4}} \|H_\delta(u_{\varepsilon,R,a})-H_g(u_{\varepsilon,R,a})\|_{(0,\alpha),[\sigma,2\sigma]} \leq C \sigma^{2+d-\mu-\frac{n}{4}}.$$

Since $d>n/4>\gamma+n/4-2$ then for $\mu=\gamma$ we have
\begin{equation}\label{eq041}
\begin{array}{c}
 \sigma^{-\gamma+1+\frac{3n}{4}}{r_\varepsilon}^{\gamma-\overline\gamma} \|H_\delta(u_{\varepsilon,R,a})-H_g(u_{\varepsilon,R,a})\|_{(0,\alpha),[\sigma,2\sigma]} \leq C {r_\varepsilon}^{d-\frac{n}{4}-l}{r_\varepsilon}^{2+l-\overline\gamma}
\end{array}
\end{equation}
and for $\mu=\overline\gamma$ we have
\begin{equation}\label{eq042}
 \sigma^{-\overline\gamma+1+\frac{3n}{4}} \|H_\delta(u_{\varepsilon,R,a})-H_g(u_{\varepsilon,R,a})\|_{(0,\alpha),[\sigma,2\sigma]} \leq C {r_\varepsilon}^{d-\frac{n}{4}-l}{r_\varepsilon}^{2+l-\overline\gamma}.
\end{equation}

Now, from (\ref{eq039}) we obtain
\begin{equation}\label{eq044}
 \begin{array}{c}
  \left\|\displaystyle\int_0^1\frac{d}{dt} Q_{\varepsilon,R,a}(v_2+t(v_1-v_2))dt \right\|_{(0,\alpha),[\sigma,2\sigma]}\\
  
\leq C\sigma^{-2-\frac{n}{2}}\left(\|v_1\|_{(2,\alpha),[\sigma,2\sigma]}+\|v_2\|_{(2,\alpha),[\sigma,2\sigma]}\right)\|v_1-v_2\|_{(2,\alpha),[\sigma,2\sigma]}.
 \end{array}
\end{equation}
Thus, using that $|h|\leq C{r_\varepsilon}^{-\overline\gamma-1}|x|^{\overline\gamma+3}+C{r_\varepsilon}^{-\overline\gamma+1}|x|^{\overline\gamma+1}$, we have
$$\sigma^{-\mu+1+\frac{3n}{4}} \|Q_{\varepsilon,R,a}({r_\varepsilon}^{-\overline\gamma}|x|^{\overline\gamma}v_\phi+h)\|_{(0,\alpha),[\sigma,2\sigma]} \leq $$
$$\begin{array}{ll}
   \leq C\left( {r_\varepsilon}^{3-\mu+\frac{n}{4}+2l-2\delta_1}+{r_\varepsilon}^{3-\mu+\frac{n}{4}}\right)\leq C{r_\varepsilon}^{1+\overline\gamma-\mu+\frac{n}{4}-l}{r_\varepsilon}^{2+l-\overline\gamma},
  \end{array}
$$
and this implies that
\begin{equation}\label{eq046}
 \|Q_{\varepsilon,R,a}({r_\varepsilon}^{-\overline\gamma}|x|^{\overline\gamma}v_\phi+h)\|_{(0,\alpha),(\gamma-1-\frac{3n}{4},\overline\gamma-1-\frac{3n}{4})}\leq C{r_\varepsilon}^{1+\frac{n}{4}-l}{r_\varepsilon}^{2+l-\overline\gamma}.
\end{equation}

Now, by (\ref{eq037}), we have that
$$\|L_g^{u_{\varepsilon,R,a}}({r_\varepsilon}^{-\overline\gamma}|x|^{\overline\gamma}v_\phi+h)\|_{(0,\alpha),[\sigma,2\sigma]} $$
$$
\leq C\left(\sigma^{-4}\|u_{\varepsilon,R,a}\|^3_{(2,\alpha),[\sigma,2\sigma]}+\|u_{\varepsilon,R,a}^{\frac{3n+4}{n-4}}\|_{(2,\alpha),[\sigma,2\sigma]}\right)\times$$
$$(\sigma^{\overline\gamma+2}{r_\varepsilon}^{-\overline\gamma+l-\delta_1}+\sigma^{\overline\gamma+1} {r_\varepsilon}^{-\overline\gamma+1}+\sigma^{\overline\gamma+3}{r_\varepsilon}^{-\overline\gamma-1})$$

Note that, by Corollary \ref{cor02}, we obtain that
$$\|u_{\varepsilon,R,a}\|_{(2,\alpha),[\sigma,2\sigma]}\leq \|u_{\varepsilon,R}\|_{(2,\alpha),[\sigma,2\sigma]}+C|a|\sigma^{2-\frac{n}{4}}$$

If $r_\varepsilon^{1+\lambda}\leq |x|\leq r_\varepsilon$ with $\lambda>0$, then
$$(1-s)\log\varepsilon+\log(2+2b)^{\frac{4}{4-n}}\leq\log(|x|^{-1} R)\leq(1-s(1+\lambda))\log\varepsilon+\log(2+2b)^{\frac{4}{4-n}}<0.$$
Thus, using (\ref{eq002}), we get
$$v_\varepsilon(-\log|x|+\log R)\leq\varepsilon^{\frac{n-4}{4}s}(2+2b).$$

Therefore,
$$u_{\varepsilon,R}(x)\leq C|x|^{\frac{4-n}{4}}r_\varepsilon^{\frac{n-4}{4}}(2+2b)$$
and so
$$\|u_{\varepsilon,R,a}\|_{(2,\alpha),[\sigma,2\sigma]}\leq C\sigma^{1-\frac{n}{4}}(r_\varepsilon^{\frac{n}{4}-1}+|a|\sigma)\leq C\sigma^{1-\frac{n}{4}+\delta_2}.$$
This implies that
$$\sigma^{-\mu-3+\frac{3n}{4}}\|u_{\varepsilon,R,a}\|_{(2,\alpha),[\sigma,2\sigma]}^3(\sigma^{\overline\gamma+2}{r_\varepsilon}^{-\overline\gamma+l-\delta_1}+\sigma^{\overline\gamma+1}{r_\varepsilon}^{-\overline\gamma+1}+\sigma^{\overline\gamma+3}{r_\varepsilon}^{-\overline\gamma-1})$$
$$\leq C(\sigma^{3\delta_2-\mu+\overline\gamma+2}{r_\varepsilon}^{-\overline\gamma+l-\delta_1}+\sigma^{3\delta_2-\mu+\overline\gamma+1}{r_\varepsilon}^{-\overline\gamma+1}+\sigma^{3\delta_2-\mu+\overline\gamma+3}{r_\varepsilon}^{-\overline\gamma-1}).$$

For $\mu=\gamma$, we obtain
$$\sigma^{-\gamma-3+\frac{3n}{4}}{r_\varepsilon}^{\gamma-\overline\gamma}\|u_{\varepsilon,R,a}\|_{(2,\alpha),[\sigma,2\sigma]}^3(\sigma^{\overline\gamma}{r_\varepsilon}^{-\overline\gamma}\|v_\phi\|_{(2,\alpha),[\sigma,2\sigma]}+\|h\|_{(2,\alpha),[\sigma,2\sigma]})$$
$$\leq C({r_\varepsilon}^{3\delta_2-\delta_1}+{r_\varepsilon}^{3\delta_2-l}){r_\varepsilon}^{2+l-\overline\gamma}.$$
since $3\delta_2>\max\{\delta_1,l\}$. For $\mu=\overline\gamma$, we have an analogous inequality.

If $0\leq \sigma\leq {r_\varepsilon}^{1+\lambda}$, then
$$\sigma^{-\mu-3+\frac{3n}{4}}\|u_{\varepsilon,R,a}\|_{(2,\alpha),[\sigma,2\sigma]}^3(\sigma^{\overline\gamma+2}{r_\varepsilon}^{-\overline\gamma+l-\delta_1}+\sigma^{\overline\gamma+1}{r_\varepsilon}^{-\overline\gamma+1}+\sigma^{\overline\gamma+3}{r_\varepsilon}^{-\overline\gamma-1})$$
$$\leq C(\sigma^{\overline\gamma-\mu+2}{r_\varepsilon}^{-\overline\gamma+l-\delta_1}+\sigma^{\overline\gamma-\mu+1}{r_\varepsilon}^{-\overline\gamma+1}+\sigma^{\overline\gamma-\mu+3}{r_\varepsilon}^{-\overline\gamma-1}),$$
which implies that for $\mu=\gamma$, we have
$$\sigma^{-\gamma-3+\frac{3n}{4}}{r_\varepsilon}^{\gamma-\overline\gamma}\|u_{\varepsilon,R,a}\|_{(2,\alpha),[\sigma,2\sigma]}^3(\sigma^{\overline\gamma+2}{r_\varepsilon}^{-\overline\gamma+l-\delta_1}+\sigma^{\overline\gamma+1}{r_\varepsilon}^{-\overline\gamma+1}+\sigma^{\overline\gamma+3}{r_\varepsilon}^{-\overline\gamma-1})$$
$$\leq  C({r_\varepsilon}^{(\overline\gamma-\gamma+2)\lambda-\delta_1}+{r_\varepsilon}^{(\overline\gamma-\gamma+1)\lambda-l}+{r_\varepsilon}^{(\overline\gamma-\mu+3)\lambda-l}){r_\varepsilon}^{2+l-\overline\gamma}.$$
and a analogous inequality for $\mu=\overline\gamma$. In an analogous way we estimate the term with $u_{\varepsilon,R,a}^{\frac{3n+4}{n-4}}$.

Therefore,
\begin{equation}\label{eq045}
 \|L_g^{u_{\varepsilon,R,a}}({r_\varepsilon}^{-\overline\gamma}|x|^{\overline\gamma}v_\phi+h)\|_{(0,\alpha),(\gamma,\overline\gamma)} \leq C{r_\varepsilon}^{c}{r_\varepsilon}^{2+l-\overline\gamma},
\end{equation}
with $c>0$. From (\ref{eq041}), (\ref{eq042}), (\ref{eq046}), (\ref{eq045}) and using inequality (\ref{eq047}) given by Corollary \ref{cor001}, we get that
\begin{equation}\label{eq050}
 \|N_\varepsilon(R,a,\phi,0 )\|_{(2,\alpha),(\gamma,\overline\gamma)} \leq  \frac{1}{2}\tau r_\varepsilon^{2+l-\overline\gamma},
\end{equation}
for $\varepsilon>0$ small enough. 

Now, by (\ref{eq037}) and using the inequality (\ref{eq048}), we find
$$\sigma^{-\mu+1+\frac{3n}{4}} \left\|L_{\varepsilon,R,a}(v_1-v_2)- L_{g}^{u_{\varepsilon,R,a}}(v_1-v_2)\right\|_{(0,\alpha),[\sigma,2\sigma]}$$
$$\leq C\sigma^{d+1-\mu}\|v_1-v_2 \|_{(2,\alpha),[\sigma,2\sigma]}\leq \sigma^{d+1+\gamma-\mu}\|v_1-v_2 \|_{(2,\alpha),(\gamma,\overline\gamma)},$$
and since $1+d+\gamma-\overline\gamma>0$, this implies that
\begin{equation}\label{eq049}
\begin{array}{c}
 \left\|L_{\varepsilon,R,a}(v_1-v_2)- L_{g}^{u_{\varepsilon,R,a}}(v_1-v_2)\right\|_{(0,\alpha),(\gamma-1-\frac{3n}{4},\overline\gamma-1-\frac{3n}{4})}\\
 \leq C{r_\varepsilon}^{1+d+\gamma-\overline\gamma}\|v_1-v_2 \|_{(2,\alpha),(\gamma,\overline\gamma)}.
\end{array}
\end{equation}

Now, using (\ref{eq044}) and the fact that 
$$\|v^\top\|_{(2,\alpha),[\sigma,2\sigma]}\leq \sigma^\gamma {r_\varepsilon}^{2+l-\gamma} \;\mbox{ and }\;\|v^\perp\|_{(2,\alpha),[\sigma,2\sigma]}\leq \sigma^\gamma {r_\varepsilon}^{2+l-\gamma},$$ 
for any $v\in C^{2,\alpha}_{(\gamma,\overline\gamma)}(B_{r_\varepsilon}(0)\backslash\{0\})$, we get

$$   \left\|\displaystyle\int_0^1\frac{d}{dt} Q_{\varepsilon,R,a}({r_\varepsilon}^{-\overline\gamma}|x|^{\overline\gamma}v_\phi+h+v_2+t(v_1-v_2))dt \right\|_{(0,\alpha),[\sigma,2\sigma]}$$
$$\leq C\left(\sigma^{\overline\gamma+\gamma-\frac{n}{2}}{r_\varepsilon}^{l-\overline\gamma-\delta_1}+\sigma^{\gamma+\overline\gamma-1-\frac{n}{2}}{r_\varepsilon}^{1-\overline\gamma}+\sigma^{\gamma+\overline\gamma+1-\frac{n}{2}}{r_\varepsilon}^{-1-\overline\gamma}\right.$$
$$\left.+\sigma^{-2-\frac{n}{2}+2\gamma}{r_\varepsilon}^{2-\gamma+l}\right)\|v_1-v_2\|_{(2,\alpha),(\gamma,\overline\gamma)}.$$

Therefore, we obtain
$$ \sigma^{-\mu+1+\frac{3n}{4}}   \left\|\displaystyle\int_0^1\frac{d}{dt} Q_{\varepsilon,R,a}({r_\varepsilon}^{-\overline\gamma}|x|^{\overline\gamma}v_\phi+h+v_2+t(v_1-v_2))dt \right\|_{(0,\alpha),[\sigma,2\sigma]}$$
$$\leq  C{r_\varepsilon}^{1+\gamma-\mu+\frac{n}{4}}\|v_1-v_2\|_{(2,\alpha),(\gamma,\overline\gamma)} $$
which implies the inequality for $\mu=\gamma$ and $\mu=\overline\gamma$.

From this and (\ref{eq049}) we obtain
\begin{equation}\label{eq051}
 \|H_\varepsilon(R,a,\phi,v_1)- H_\varepsilon(R,a,\phi,v_2) \|_{(2,\alpha),(\gamma,\overline\gamma)}\leq \frac{1}{2} \|v_1-v_2\|_{(2,\alpha),(\gamma,\overline\gamma)},
\end{equation}
for $\varepsilon>0$ small enough. Therefore, using (\ref{eq050}) and (\ref{eq051}) we obtain the result.
\end{proof}
We summarize the main result of this section in the next theorem.
\begin{theorem}\label{teo01}
Let $\gamma=\delta_{n,2}+1-n/4-\varepsilon_1$ and  $\overline{\gamma}=n/4+1+\varepsilon_1$ be constants, where $\varepsilon_1>0$ is a small constant. There exists a constant $\varepsilon\in(0,1)$ such that for each $\varepsilon\in(0,\varepsilon_0)$, $\kappa>0$, $\tau>0$, $|b|<1/2$, $a\in\mathbb{R}^n$,  $\delta_1,\delta_2,l\in\mathbb R_+$ small, and $\phi\in C^{2,\alpha}(\mathbb{S}_{r_\varepsilon}^{n-1})^\perp$ with $3\delta_2>\max\{\delta_1,l\}$, $|a|r_\varepsilon^{1-\delta_2}\leq 1$ and $\|\phi\|_{(2,\alpha),r_\varepsilon}\leq \kappa r_\varepsilon^{2+l-\delta_1}$, there exists a solution $U_{\varepsilon,R,a,\phi}\in C^{2,\alpha}_{(\gamma,\overline\gamma)}(B_{r}(0)\backslash\{0\})$ for the equation
$$\left\{\begin{array}{lcl}
H_g(u_{\varepsilon,R,a}+r_\varepsilon^{-\overline\gamma}|x|^{\overline\gamma}v_{\phi}+h+ U_{\varepsilon,R,a,\phi})=0 & \mbox{ in } & B_{r_\varepsilon}(0)\backslash\{0\}\\
\pi''_{r_\varepsilon}((r_\varepsilon^{-\overline\gamma}|x|^{\overline\gamma}v_{\phi}+ U_{\varepsilon,R,a,\phi})|_{\partial B_{r_\varepsilon}(0)})=\phi & \mbox{ on } & \partial B_{r_\varepsilon}(0)
\end{array}\right.,$$
such that
\begin{equation}\label{eq71}
\|U_{\varepsilon,R,a,\phi}\|_{(2,\alpha),(\gamma,\overline\gamma),r_\varepsilon}\leq \tau r_\varepsilon^{2+l-\overline\gamma}
\end{equation}
and
\begin{equation}\label{eq82}
\|U_{\varepsilon,R,a,\phi_1}- U_{\varepsilon,R,a,\phi_2}\|_{(2,\alpha),(\gamma,\overline\gamma),r_\varepsilon}\leq Cr_\varepsilon^{\delta_4-\overline\gamma} \|\phi_1-\phi_2\|_{(2,\alpha),r},
\end{equation}
for some constant $\delta_4>0$ and $p>0$ is small.
\end{theorem}
\begin{proof}
The solution $U_{\varepsilon,R,a,\phi}$ is the fixed point of the map $N_\varepsilon(R,a,\phi,\cdot)$ given by Proposition \ref{propo03} with the estimate (\ref{eq71}).

Using the fact that $U_{\varepsilon,R,a,\phi}$ is a fixed point of the map $N_\varepsilon(R,a,\phi,\cdot)$ we can show that
$$\|U_{\varepsilon,R,a,\phi_1}- U_{\varepsilon,R,a,\phi_2}\|_{(2,\alpha),(\gamma,\overline{\gamma})}$$
$$\leq 2\|N_\varepsilon(R,a,\phi_1, U_{\varepsilon,R,a,\phi_2})- N_\varepsilon(R,a,\phi_2, U_{\varepsilon,R,a,\phi_2})\|_{(2,\alpha),(\gamma,\overline{\gamma})}$$
$$\begin{array}{l}
   \leq  C \left\| L_g^{u_{\varepsilon,R,a}}(r_\varepsilon^{-\overline\gamma}|x|^{\overline\gamma}v_{\phi_1-\phi_2})\right \|_{(0,\alpha),(\gamma-1-\frac{3n}{4},\overline{\gamma}-1-\frac{3n}{4})}\\
   \\
   +\left\| \displaystyle\int_0^1\frac{d}{dt}
Q_{\varepsilon,R,a}(U_{\varepsilon,R,a,\phi_2}+r_\varepsilon^{-\overline\gamma}|x|^{\overline\gamma}v_{\phi_1+t(\phi_2-\phi_1)}+h)dt\right \|_{(0,\alpha),(\gamma-1-\frac{3n}{4},\overline{\gamma}-1-\frac{3n}{4})}
  \end{array}
$$

From this and the estimates given by the proof of the Proposition \ref{propo03} it follows (\ref{eq82}).
\end{proof}
We will write the full conformal factor of the resulting constant scalar  curvature metric with respect to the metric $g$ as
\begin{equation*}\label{eq019}
\mathcal{A}_{\varepsilon}(R,a,\phi):=u_{\varepsilon,R,a}+ r_{\varepsilon}^{-\overline\gamma}|x|^{\overline\gamma}v_\phi+h+U_{\varepsilon,R,a,\phi},
\end{equation*}
in conformal normal coordinates. The previous analysis says that the metric 
\begin{equation*}\label{eq020}
\hat{g}=\mathcal{A}_{\varepsilon}(R,a,\phi)^{\frac{8}{n-4}}g
\end{equation*}
is defined in $\overline{B_{r_\varepsilon}(p)}\backslash\{p\}\subset M$, it is complete and has $\sigma_2(A_{\hat{g}})=\frac{n(n-1)}{8}$. The completeness follows from the estimate $\mathcal{A}_{\varepsilon}(R,a,\phi)\geq c|x|^{\frac{4-n}{4}},$
for some constant $c>0$.

\section{Exterior Analysis}\label{sec04}
\subsection{Analysis in $M\backslash B_r(p)$}\label{sec10}
In contrast with the previous section, in which we worked with conformal normal coordinates, in this section it is better to work with the constant $\sigma_2-$curvature metric, since in this case the constant function 1 satisfies $H_{g_0}(1)=0$. Hence, in this section $(M^n,g_0)$ is an $n-$dimensional nondegenerate closed Riemannian manifold with $\sigma_2(A_{g_0})=\frac{n(n-1)}{8}$ which is conformal to some $2-$admissible metric. From this and using (\ref{eq010}) we find that $R_{g_0}>0$ in $M$.

Let $r_1\in(0,1)$ and $\Psi:B_{r_1}(0)\rightarrow M$ be a coordinate system with respect to $g=\mathcal F^{\frac{8}{n-4}}g_0$ on $M$ cetered at $p$, where $\mathcal F$ is defined in Section \ref{sec02} and it satisfies $\mathcal F=1+O(|x|^2)$ in $g-$normal coordinates. Since $g_{ij}=1+O(|x|^2)$ then $(g_0)_{ij}=\delta_{ij}+O(|x|^2)$.

We start this section remember a result from \cite{M} (see also \cite{MR2337310} and \cite{MR2639545}).
\begin{proping}\label{propo002}
 Assume that $\varphi\in C^{2,\alpha}(\mathbb S_r^{n-1})$ and let $\mathcal Q_r(\varphi)$ be the only solution of 
 $$\left\{\begin{array}{lcl}
           \Delta\mathcal Q_r(\varphi_r)=0 & \mbox{ in } & \mathbb R^n\backslash B_r(0)\\
           \mathcal Q_r(\varphi_r)=\varphi_r & \mbox{ on } & \partial B_r(0)
          \end{array}
\right.$$
which tends to $0$ at $\infty$. Then
$$\|\mathcal Q_r(\varphi_r)\|_{C_{1-n}^{2,\alpha}(\mathbb R^n\backslash B_r(0))}\leq Cr^{n-1}\|\varphi_r\|_{(2,\alpha),r},$$
if $\varphi$ is $L^2-$orthogonal to the constant function. Moreover, if $\varphi=\displaystyle\sum_{j=1}^\infty\varphi_i,$
where $\varphi_i$ belonging to the eigenspace associated to the eigenvalue $i(i+n-2)$, then
$$\mathcal Q_r(\varphi_r)(x)=\sum_{j=1}^\infty r^{n+j-2}|x|^{2-n-j}\varphi_i.$$
\end{proping}

Consider $r>0$ much smaller than $r_1$. Let $\varphi\in C^{2,\alpha}(\mathbb S_r^{n-1})$ be a function $L^2-$orthogonal to the constant functions. Remember that $M_{s}=M\backslash \Psi(B_s(0))$. Let $u_\varphi\in C_\nu^{2,\alpha}(M_r)$ be a function such that $u_\varphi\equiv 0$ in $M_{r_1}$ and $u_\varphi\circ\Psi=r^{-\overline\gamma}|x|^{\overline\gamma}\mathcal Q_r(\pi_r''(\varphi))\eta+\eta\mathcal Q_r(\varphi-\pi_r''(\varphi))$, where $\eta$ is a smooth, radial function equal to 1 in $B_{3r}(0)$, vanishing in $\mathbb R^n\backslash B_{4r}(0)$, and satisfying $|\partial_r\eta(x)|\leq c|x|^{-1}$ and $|\partial_r^2\eta(x)|\leq c|x|^{-2}$ for all $x\in B_{r_1(0)}$. This implies that $\|\eta\|_{(2,\alpha),[\sigma,2\sigma]}\leq c$, for every $r\leq \sigma\leq r_1$. Hence, we get
\begin{equation}\label{eq031}
 \|u_\varphi\|_{C_\nu^{2,\alpha}(M_r)}\leq cr^{-\nu}\|\varphi\|_{(2,\alpha),r},
\end{equation}
for all $\nu\geq 1-n.$ Finally, define a function $G_p\in C^\infty(M\backslash\{p\})$ by $G_p\circ\Psi=\eta|x|^{2-\frac{n}{2}}$ in $B_{r_1}(0)$ and equal to zero in $M_{r_1}$.

Our goal in this section is to solve the equation
\begin{equation}\label{eq25}
H_{{g_0}}(1+\Lambda G_p+u_\varphi+v)=0\;\;\;\mbox{ on }\;\;\;M\backslash B_r(p),
\end{equation}
for some $r>0$, $\Lambda\in\mathbb R $ and $\varphi\in C^{2,\alpha}(\mathbb S^{n-1}_r)$, with $1+\Lambda G_p+u_\varphi+v>0$, where $H_{g_0}$ is defined in (\ref{eq014}). 

To solve this equations we linearize $H_{{g_0}}$ about 1 to get
\begin{equation}\label{eq029}
 H_{{g_0}}(1+\Lambda G_p+u_\varphi+v)=L_{{g_0}}^1(\Lambda G_p+u_\varphi+v)+Q(\Lambda G_p+u_\varphi+v),
\end{equation}
since $H_{{g_0}}(1)=0$, where $L_{g_0}^1$ is defined in (\ref{eq030}) and
\begin{equation}\label{eq018}
Q(u)=\int_0^1\int_0^1\frac{d}{ds}L_{g_0}^{1+tsu}(u)dsdt.
\end{equation}

Therefore, if $L_{g_0}^1$ has a right inverse $G_{r,g_0}$, then by (\ref{eq029}), to finding a solution of the equation (\ref{eq25}) it is enough to show that for suitable $\Lambda\in\mathbb{R}$, and $\varphi\in C^{2,\alpha}(\mathbb{S}^{n-1}_r)$ the map $\mathcal{M}_r(\Lambda,\varphi,\cdot):C^{2,\alpha}_\nu(M_r)\rightarrow C^{2,\alpha}_\nu(M_r)$, given by
\begin{equation}\label{eq60}
\mathcal{M}_r(\Lambda,\varphi,v) = -G_{r,{{g_0}}}\left(Q(\Lambda G_{p} +u_{\varphi}+v)+L_{g_0}^1(\Lambda G_{p}+u_{\varphi})\right),
\end{equation}
has a fixed point for small enough $r>0$. 
\subsection{Inverse for $L_{{g_0}}^1$ in $M_r$}\label{sec15}
To find a right inverse for $L_{{g_0}}^1$, we will follow the method of Jleli in \cite{M}.

\begin{lemma}[See \cite{M} and \cite{MR2194146}]\label{lem3}
Assume that $\nu\in\left(1-n,2-n\right)$ is fixed and that $0<2r<s\leq r_1$. Then there exists an operator
$$\tilde{G}_{r,s}:C^{0,\alpha}_{\nu-2}(\Omega_{r,s})\rightarrow C^{2,\alpha}_{\nu}(\Omega_{r,s})$$
such that, for all $f\in C^{0,\alpha}_{\nu}(\Omega_{r,s})$, the function $w=\tilde{G}_{r,s}(f)$ is a solution of
\begin{eqnarray*}
\left\{\begin{array}{rclccl}
\Delta w & = & f & \mbox{ in } & B_s(0)\backslash B_r(0)\\
w & = & 0 & \mbox{ on } & \partial B_s(0)\\
w & \in & \mathbb{R} & \mbox{ on } & \partial B_r(0)
\end{array}\right..
\end{eqnarray*}
In addition,
$$\|\tilde{G}_{r,s}(f)\|_{C_\nu^{2,\alpha}(\Omega_{r,s})}\leq C\|f\|_{C_{\nu-2}^{0,\alpha}(\Omega_{r,s})},$$
for some constant $C>0$ that does not depend on $s$ and $r$.
\end{lemma}

Note that, since $R_{g_0}\not=0$, in the previous lemma we can consider $-\frac{n-4}{8(n-2)}R_{g_0}\Delta$ instead of $\Delta$. By a perturbation argument we obtain the next lemma.

\begin{lemma}
Assume that $\nu$  and $\eta>0$ is fixed with $\nu$ and $\nu-\eta$ in $(1-n,2-n)$. Let $0<2r<s\leq r_1$ be constants. Then, for $r_1$ small enough, there exists an operator
$$G_{r,s}:C^{0,\alpha}_{\nu-2}(\Omega_{r,s})\rightarrow C^{2,\alpha}_\nu(\Omega_{r,s})$$
such that, for all $f\in C^{0,\alpha}_{\nu}(\Omega_{r,s})$, the function $w=G_{r,s}(f)$ is a solution of
\begin{eqnarray*}
\left\{\begin{array}{rclccl}
L_{g_0}^1 w & = & f & \mbox{ in } & B_s(0)\backslash B_r(0)\\
w & = & 0 & \mbox{ on } & \partial B_s(0)\\
w & \in & \mathbb{R} & \mbox{ on } & \partial B_r(0)
\end{array}\right..
\end{eqnarray*}
In addition,
\begin{equation}\label{eq053}
\|G_{r,s}(f)\|_{C_\nu^{2,\alpha}(\Omega_{r,s})}\leq Cr^{-\eta}\|f\|_{C_{\nu-2}^{0,\alpha}(\Omega_{r,s})}, 
\end{equation}
for some constant $C>0$ that does not depend on $s$ and $r$.
\end{lemma}
\begin{proof}
Note that  by (\ref{eq030}) we get
$$\begin{array}{rcl}
L_{g_0}^1(\tilde{G}_{r,s}(v))-v & = & -\displaystyle\frac{n-4}{8(n-2)}R_{g} (\Delta_{g_0}-\Delta)\tilde{G}_{r,s}(v)- \frac{n(n-1)(n-4)}{8}\tilde{G}_{r,s}(v)\\
& + & \displaystyle\frac{n-4}{4(n-2)}\langle Ric_{g_0},\nabla_{g_0}^2\tilde{G}_{r,s}(v)\rangle_{g_0}.
\end{array}$$
Thus, from
$$ \sigma^{2-\nu+\eta} \|L_{g_0}^1(\tilde{G}_{r,s}(v))-v \|_{(0,\alpha),[\sigma,2\sigma]}\leq C\sigma^{-\nu+\eta} \|\tilde{G}_{r,s}(v)\|_{(2,\alpha),[\sigma,2\sigma]}$$
we get
$$ \|L_{g_0}^1(\tilde{G}_{r,s}(v))- v\|_{C^{0,\alpha}_{\nu-2-\eta} (\Omega_{r,s})}\leq Cs^{\eta} \|\tilde{G}_{r,s}(v)\|_{C^{2,\alpha}_{\nu} (\Omega_{r,s})}\leq Cs^\eta \|v\|_{C^{0,\alpha}_{\nu-2} (\Omega_{r,s})}.$$

Therefore, for $s>0$ small enough there is an inverse $(L_{g_0}^1\circ\tilde{G}_{r,s})^{-1}: C^{0,\alpha}_{\nu-2} (\Omega_{r,s})\rightarrow C^{2,\alpha}_{\nu-2-\eta} (\Omega_{r,s})$ with bounded norm. 
Besides, the operator $\tilde{G}_{r,s}:C^{0,\alpha}_{\nu-2-\eta} (\Omega_{r,s})\rightarrow C^{2,\alpha}_{\nu} (\Omega_{r,s})$ satisfies the condition
$$\|\tilde{G}_{r,s}(f)\|_{C^{2,\alpha}_\nu(\Omega_{r,s})}\leq Cr^{-\eta}\|f\|_{C^{0,\alpha}_{\nu-2-\eta}(\Omega_{r,s})}.$$

Therefore, we have the right inverse $G_{r,s}:=\tilde{G}_{r,s}\circ (L_{g_0}^1\circ\tilde{G}_{r,s})^{-1}: C^{0,\alpha}_{\nu-2} (\Omega_{r,s})\rightarrow C^{2,\alpha}_{\nu} (\Omega_{r,s})$, with the norm estimate (\ref{eq053}).
\end{proof}
\begin{theorem}\label{teo03}
Assume that $\nu$  and $\eta>0$ is fixed with $\nu$ and $\nu-\eta$ in $(1-n,2-n)$. There exists $r_2< \frac{1}{4}r_1$, such that, for all $r\in(0,r_2)$ we can define an operator
$G_{r,{{g_0}}}:C^{0,\alpha}_{\nu-2}(M_r)\rightarrow C^{2,\alpha}_{\nu}(M_r),$
with the property that, for all $f\in C^{0,\alpha}_{\nu-2}(M_r)$ the function $w=G_{r,{{g_0}}}(f)$ solves
$$L_{g_0}^1(w)  =  f,$$
in $M_r$ with $w\in\mathbb{R}$ constant on $\partial B_r(p)$. In addition
$$\|G_{r,{{g_0}}}(f)\|_{C^{2,\alpha}_\nu(M_r)}\leq Cr^{-\eta}\|f\|_{C^{0,\alpha}_{\nu-2}(M_r)},$$
where $C>0$ does not depend on $r.$
\end{theorem}
\begin{proof} The proof is analogous to the proof of Proposition 13.28 in \cite{M}. 

We can take $s=r_1$ with $r_1>0$ small enough. Let $f\in C^{0,\alpha}_{\nu-2}(M_r)$ and define a function $w_0\in C^{2,\alpha}_\nu(M_r)$ by $w_0:=\chi_1 G_{r,r_1}(f|_{\Omega_{r,r_1}})$ where $\chi_1$ is a smooth, radial function equal to 1 in $B_{\frac{1}{2}r_1}(p)$, vanishing in $M_{r_1}$ and satisfying $|\partial_r\chi_1(x)|\leq c|x|^{-1}$ and $|\partial_r^2\chi_1(x)|\leq c|x|^{-2}$ for all $x\in B_{r_1}(0)$. From this it follows that $\|\chi_1\|_{(2,\alpha),[\sigma,2\sigma]}$ is uniformly bounded in $\sigma$, for every $r\leq \sigma\leq \frac{1}{2}r_1$. Thus,
$$
\sigma^{-\nu}\|w_0\|_{(2,\alpha),[\sigma,2\sigma]} \leq C\| G_{r,r_1}(f|_{\Omega_{r,r_1}}) \|_{C^{2,\alpha}_\nu (\Omega_{r,r_1})}\leq Cr^{-\eta} \| f\|_{{C^{0,\alpha}_{\nu-2}(M_r)}}.$$
Since $w_0$ vanishes in $M_{r_1}$, then we get
\begin{equation}\label{eq08}
\|w_0\|_{C^{2,\alpha}_\nu(M_r)}\leq Cr^{-\eta} \| f\|_{{C^{0,\alpha}_{\nu-2}(M_r)}},
\end{equation}
where the constant $C>0$ is independent of $r$ and $r_1$.

Since $w_0=G_{r,r_1}(f|_{\Omega_{r,r_1}})$ in $\Omega_{r,\frac{1}{2}r_1}$, the function
\begin{equation}\label{eq060}
 h:=f-L_{{g_0}}^1(w_0)
\end{equation}
is supported in $M_{\frac{1}{2}r_1}$. We can consider that $h$ is defined on the whole $M$ with $h\equiv 0$ in $B_{\frac{1}{2}r_1}(p)$, and using that $L_{g_0}^1$ is bounded, we get that
$$\begin{array}{rcl}
\|h\|_{C^{0,\alpha}(M)} & = & \|h\|_{C^{0,\alpha}(M_{\frac{1}{2}r_1})}\leq C_{r_1}\|h\|_{C_{\nu-2}^{0,\alpha}(M_r)}\\
& \leq & C_{r_1} (\|f\|_{C^{0,\alpha}_{\nu-2}(M_r)}+ \|w_0\|_{C_{\nu}^{2,\alpha}(M_r)}).
\end{array}$$
From (\ref{eq08})
\begin{equation}\label{eq23}
\|h\|_{C^{0,\alpha}(M)}\leq C_{r_1}r^{-\eta} \|f\|_{C^{0,\alpha}_{\nu-2}(M_r)},
\end{equation}
with the constant $C_{r_1}>0$ independent of $r$.

Since $L_{{g_0}}^1:C^{2,\alpha}(M)\rightarrow C^{0,\alpha}(M)$ has a bounded inverse, we can define the function $w_1:=\chi_2 (L_{{g_0}}^1)^{-1}(h),$
where $\chi_2$ is a smooth, radial function equal to 1 in $M_{2r_2}$, vanishing in $B_{r_2}(p)$ and satisfying $|\partial_r\chi_2(x)|\leq c|x|^{-1}$ and $|\partial_r^2\chi_2(x)|\leq c|x|^{-2}$ for all $x\in B_{2r_2}(0)$ and some $r_2\in(r,\frac{1}{4}r_1)$ to be chosen later. This implies that $\|\chi_2\|_{(2,\alpha),[\sigma,2\sigma]}$ is uni\nolinebreak formly bounded in $\sigma$, for every $r\leq \sigma\leq \frac{1}{2}r_1$.

Hence, from (\ref{eq23}) and the fact that $(L_{g_0}^1)^{-1}$ is bounded, we have that
\begin{equation}\label{eq19}
\|w_1\|_{C^{2,\alpha}_\nu(M_r)}\leq C_{r_1}r^{-\eta}  \|f\|_{C^{0,\alpha}_{\nu-2}(M_r)},
\end{equation}
since $\nu<0$, where the constant $C_{r_1}>0$ is independent of $r$ and $r_2$.

Now, define an application $F_{r,{{g_0}}}:C^{0,\alpha}_{\nu-2}(M_r)\rightarrow C^{2,\alpha}_{\nu}(M_r)$ as $F_{r,{{g_0}}}(f):=w_0+w_1.$ From (\ref{eq08}) and (\ref{eq19}) we obtain
\begin{equation}\label{eq31}
\|F_{r,g_0}(f)\|_{C^{2,\alpha}_{\nu}(M_r)}\leq C_{r_1}r^{-\eta} \|f\|_{C^{0,\alpha}_{\nu-2}(M_r)},
\end{equation}
where the constant $C_{r_1}>0$ does not depend on $r$ and $r_2$.

Therefore we get the following
\begin{enumerate}
\item[i)] In $\Omega_{r,r_2}$ we have $w_0=G_{r,r_1}(f|_{\Omega_{r,r_1}})$ and $w_1=0$. Therefore $L_{g_0}^1(F_{r,g_0}(f))=f.$
\item[ii)] In $\Omega_{r_2,2r_2}$ we have $w_0=G_{r,r_1}(f|_{\Omega_{r,r_1}})$ and $w_1=\chi_2L_{g_0}^{-1}(h)$. Hence we have that $L_{g_0}^1(F_{r,g_0}(f))=f+L_{g_0}^1(\chi_2(L_{g_0}^1)^{-1}(h)).$
\item[iii)] In $M_{2r_2}$ we have $w_1=(L_{g_0}^1)^{-1}(h)$ and by (\ref{eq060}) we get
$$L_{g_0}^1(F_{r,g_0}(f))=L_{g_0}^1(w_0)+h=f.$$
\end{enumerate}

Hence, using the boundedness of $L_{g_0}^1$ and (\ref{eq23}) we get
$$\begin{array}{rcl}
\|L_{g_0}^1(F_{r,g_0}(f))-f\|_{(0,\alpha),[\sigma,2\sigma]} & \leq & \|L_{g_0}^1(\chi_2(L_{g_0}^1)^{-1}(h))\|_{(0,\alpha),[\sigma,2\sigma]}\\
& \leq & C_{r_1}r_2^{-3}r^{-\eta} \|f\|_{C^{0,\alpha}_{\nu-2}(M_r)},
\end{array}$$
where the constant $C_{r_1}>0$ does not depend on $r$.

Therefore
\begin{equation*}\label{eq30}
\|L_{g_0}^1(F_{r,g_0}(f))-f\|_{C^{0,\alpha}_{\nu-2}(M_r)}\leq C_{r_1}r^{-\eta} r_2^{-1-
\nu}\|f\|_{C^{0,\alpha}_{\nu-2}(M_r)}
\end{equation*}
since $1-n<\nu<2-n$ implies that $2-\nu>0$ and $-1-\nu>0$, for some constant $C_{r_1}>0$ independent of $r$ and $r_2$. If we consider $r_2=2r$, then
\begin{equation}\label{eq008}
\|L_{g_0}^1(F_{r,g_0}(f))-f\|_{C^{0,\alpha}_{\nu-2}(M_r)}\leq C_{r_1}r^{-1-
\nu-\eta}\|f\|_{C^{0,\alpha}_{\nu-2}(M_r)}.
\end{equation}
The assertion follows from a perturbation argument by (\ref{eq31}) and (\ref{eq008}).
\end{proof}

\subsection{Constant $\sigma_2-$curvature metrics on $M\backslash B_r(p)$}\label{sec13}

We will show that the map $\mathcal{M}_r(\Lambda,\varphi,\cdot)$, given by (\ref{eq60}) where $G_{r,g_0}$ is given by Theorem \ref{teo03} with $\eta>0$ small enough, is a contraction, and as a consequence the fixed point will depend continuously on the parameters $r$, $\Lambda$ and $\varphi$.
\begin{proping}\label{propo07}
Let $\nu\in(1-n,2-n)$ and $\eta>0$ small enough. Let $\beta$, $\gamma$, $\delta_4$, $\delta_5$ and $l$  be fixed positive constants such that $l>\max\{\delta_5,2\delta_4\}$. There exists $r_2\in (0,r_1/4)$ such that if $r\in(0,r_2)$, $\Lambda\in\mathbb{R}$ with $|\Lambda|^2\leq r^{n+l+\delta_5}$ and $\varphi\in C^{2,\alpha}(\mathbb{S}_r^{n-1})$ which is $L^2-$orthogonal to the constant functions with $\|\varphi\|_{(2,\alpha),r}\leq \beta r^{2+l-\delta_4}$, then there is a fixed point of the map $\mathcal{M}_r(\Lambda,\varphi, \cdot)$ in the ball of radius $\gamma r^{2+l-\nu}$ in $C_{\nu}^{2,\alpha}(M_r)$.
\end{proping}
\begin{proof}
From (\ref{eq60}) and Theorem \ref{teo03} it follows that
$$\|\mathcal{M}_r(\Lambda,\varphi,0) \|_{C^{2,\alpha}_\nu(M_r)}\leq Cr^{-\eta}\left(\|Q^1_{g_0}(\Lambda G_{p}+u_{\varphi}) \|_{C_{\nu-2}^{0,\alpha}(\Omega_{r,r_1})}\right.$$
$$\left.+\|L_g(\Lambda G_{p}+u_{\varphi} )\|_{C_{\nu-2}^{0,\alpha}(\Omega_{r,r_1})}\right),$$
for some constant $C>0$ independent of $r$, since the functions $G_p$, $u_\varphi$ and $h$ are equal to zero in $M\backslash B_{4r}(p)$.

By (\ref{eq018}), we get
$$\begin{array}{lcl}
   \|Q_{g}(\Lambda G_p+u_\varphi)\|_{(0,\alpha),[\sigma,2\sigma]} & \leq &  C\sigma^{-4}\|\Lambda G_p+u_\varphi\|^2_{(2,\alpha),[\sigma,2\sigma]}
  \end{array}
$$
and then, since $G_p=O(|x|^{2-\frac{n}{2}})$, $u_\varphi=O(|x|^{1-n})$ and (\ref{eq031}) holds, we have that
$$\begin{array}{rcl}
\sigma^{2-\nu}\|Q_{g}(\Lambda G_p+u_\varphi)\|_{(0,\alpha),[\sigma,2\sigma]} & \leq & C\sigma^{-2-\nu} \|\Lambda G_p+u_\varphi\|^2_{(2,\alpha),[\sigma,2\sigma]}\\
& \leq & C(r^{l-\delta_5}+ r^{l-2\delta_4})r^{2+l-\nu}
\end{array}$$

So,
$$\|Q^1_{g_0}(\Lambda G_p+u_\varphi)\|_{C^{0,\alpha}_{\nu-2}(\Omega_{r,s})} \leq r^{\delta_6} r^{2+l-\nu},$$
with $\delta_6>0$. Now, by (\ref{eq030}) we obtain that
$$ \|L_g^1(\Lambda G_p+u_\varphi)\|_{(0,\alpha),[\sigma,2\sigma]} \leq C(\sigma^{d-3}+1)\|\Lambda G_p+ u_\varphi\|_{(2,\alpha),[\sigma,2\sigma]},$$
and this implies that
$$\sigma^{2-\nu} \|L_g(\Lambda G_p+u_\varphi)\|_{(0,\alpha),[\sigma,2\sigma]} \leq C(\sigma^{d-1-\nu}+\sigma^{2-\nu})\|\Lambda G_p+ u_\varphi\|_{(2,\alpha),[\sigma,2\sigma]}$$
$$\begin{array}{rl}
& \leq  C\left(r^{d-1-\frac{l}{2}+\frac{\delta_5}{2}}+r^{2-\frac{l}{2}+\frac{\delta_5}{2}}+ r^{d-1-\delta_4}+r^{2-\delta_4}\right)r^{2+l-\nu}
\end{array}$$

Therefore
\begin{equation}\label{eq054}
 \|\mathcal{M}_r(\Lambda,\varphi,0) \|_{C^{2,\alpha}_\nu(M_r)}\leq \frac{1}{2}\gamma r^{2+l-\nu}.
\end{equation}

Now, we have
$$\|\mathcal{M}(v_1)- \mathcal{M}(v_2)\|_{C^{2,\alpha}_\nu(M_r)} $$
$$\begin{array}{l}
\leq Cr^{-\eta}\|  Q^1_{g_0}(\Lambda G_p+u_\varphi+v_1)- Q^1_{g_0}(\Lambda G_p+ u_\varphi+v_2)\|_{C^{0,\alpha}_{\nu-2}(M_r)}\\
  = Cr^{-\eta} \left(\|  Q^1_{g_0}(v_1)- Q^1_{g_0}(v_2)\|_{C^{0,\alpha}(M_{r_1})}\right.\\
  +\left.\|  Q^1_{g_0}(\Lambda G_p+u_\varphi+v_1)- Q^1_{g_0}(\Lambda G_p+u_\varphi+v_2)\|_{C^{0,\alpha}_{\nu-2}(\Omega_{r,r_1})}\right).
  \end{array}
$$
Since $r>0$ is small and $\|v_1\|_{C^{2,\alpha}_{\nu}(M_r)}<\gamma r^{2+l-\nu}$, then
\begin{equation}\label{eq055}
 \begin{array}{l}
  \|Q^1_{g_0}(v_1)- Q^1_{g_0}(v_2) \|_{C^{0,\alpha}(M_{r_1})} \\
  \hspace{3cm}\leq C(\|v_1\|_{C^{2,\alpha}_{\nu}(M_r)}+ \|v_2\|_{C^{2,\alpha}_{\nu}(M_r)}) \|v_1-v_2\|_{C^{2,\alpha}_{\nu}(M_r)}\\
  \hspace{3cm}\leq Cr^{2+l-\nu}\|v_1-v_2\|_{C^{2,\alpha}_{\nu}(M_r)}.
 \end{array}
\end{equation}

Now we have
$$\|Q^1_{g_0}(\Lambda G_p+u_\varphi+v_1)- Q^1_{g_0}(\Lambda G_p+u_\varphi+v_2) \|_{(0,\alpha),[\sigma,2\sigma]}$$
$$\leq C\sigma^{-4}\left(|\Lambda|\sigma^{2-\frac{n}{2}}+\|u_\varphi\|_{(2,\alpha),[\sigma,2\sigma]}+\sigma^{2+l}\right)\|v_1-v_2\|_{(2,\alpha),[\sigma,2\sigma]}$$
$$\leq C\left(\sigma^{-2-\frac{n}{2}}r^{\frac{n}{2}+\frac{l}{2}+\frac{\delta_5}{2}}+\sigma^{-3-n}r^{1+n+l-\delta_4}+\sigma^{-2+l}\right)\|v_1-v_2\|_{(2,\alpha),[\sigma,2\sigma]}.
$$
Then
$$\sigma^{2-\nu}\|Q^1_{g_0}(\Lambda G_p+u_\varphi+v_1)- Q^1_{g_0}(\Lambda G_p+u_\varphi+v_2) \|_{(0,\alpha),[\sigma,2\sigma]}$$
$$\leq C\left(r^{(l+\delta_5)/2}+r^{l-\delta_4}\right)\|v_1-v_2\|_{C^{2,\alpha}_{\nu}(M_r)},$$
and this together with (\ref{eq055}) implies that
\begin{equation}\label{eq115}
\|\mathcal{M}_r(\Lambda,\varphi,v_1)- \mathcal{M}_r(\lambda,\varphi,v_2)\|_{C_\nu^{2,\alpha}(M_r)}\leq \frac{1}{2}\|v_1-v_2\|_{C_\nu^{2,\alpha}(M_r)},
\end{equation}
for $r>0$ small enough. Therefore, from (\ref{eq054}) and (\ref{eq115}) we obtain the result.
\end{proof}

From Proposition \ref{propo07} we get the main result of this section.
\begin{theorem}\label{teo02}
Let $\nu\in(1-n,2-n)$ and $\eta>0$ small enough. Let $\beta$, $\gamma$, $\delta_4$, $\delta_5$ and $l$ be fixed positive constants such that $l>\{\delta_5,\delta_4\}$. There is $r_2>0$ such that if $r\in(0,r_2)$, $\Lambda\in\mathbb{R}$ with $|\Lambda|^2\leq r^{n+l+\delta_5}$ and $\varphi\in C^{2,\alpha}(\mathbb{S}_r^{n-1})$ which is $L^2-$orthogonal to the constant functions with $\|\varphi\|_{(2,\alpha),r}\leq \beta r^{2+l-\delta_4}$, then there is a solution $V_{\Lambda,\varphi}\in C^{2,\alpha}_\nu(M_r)$ to the problem
$$\left\{\begin{array}{l}
H_{{g_0}} (1+\Lambda G_{p}+ u_\varphi+V_{\Lambda,\varphi})=0\;\;\;\mbox{ in }\;\;\;M_r\\
(u_{\varphi}+V_{\Lambda,\varphi})\circ\Psi|_{\partial B_r(0)}-\varphi\in\mathbb{R}\;\;\;\mbox{ on }\;\;\;\partial M_r
\end{array}\right..$$

Moreover,
\begin{equation}\label{eq72}
\|V_{\Lambda,\varphi}\|_{C^{2,\alpha}_\nu(M_r)}\leq \gamma r^{2+l-\nu},
\end{equation}
and
\begin{equation}\label{eq86}
\|V_{\Lambda,\varphi_1}- V_{\Lambda,\varphi_2}\|_{C^{2,\alpha}_\nu(M_r)} \leq Cr^{\delta_6-\nu}\|\varphi_1-\varphi_2\|_{(2,\alpha),r},
\end{equation}
for some constant $\delta_6>0$ small enough independent of $r$.
\end{theorem}
\begin{proof} The solution $V_{\Lambda,\varphi}$ is the fixed point of $\mathcal{M}_r(\Lambda,\varphi,\cdot)$ given by Proposition \ref{propo07} with the estimate (\ref{eq72}). The inequality (\ref{eq86}) follows similarly to (\ref{eq82}), that is, it follows by the estimates obtained by the proof of the Proposition \ref{propo07}.
\end{proof}

If $g$ is the metric given in the previous section, then there is a function $f$ such that $g_0=f^{\frac{8}{n-4}}g$ and in the normal coordinate system centered at $p$ with respect to $g$ we have $f=1+O(|x|^2)$, in fact, $f=1/\mathcal F$. We will denote the full conformal factor of the resulting constant $\sigma_2-$curvature metric in $M_r$ with respect to the metric $g$ as $\mathcal{B}_r(\Lambda,\varphi)$, that is, the metric
\begin{equation}\label{eq021}
\tilde{g}=\mathcal{B}_r(\Lambda,\varphi)^{\frac{8}{n-4}}g_0
\end{equation}
has $\sigma_2(A_g)=n(n-1)/8$, where
$$\mathcal{B}_r(\Lambda,\varphi):=f+\Lambda fG_p+fu_{\varphi}+fV_{\Lambda,\varphi}.$$

\section{Gluing the initial data}\label{sec05}
By the Theorem \ref{teo01} there exists a family of constant $\sigma_2-$curvature metrics in  $\overline{B_{r_\varepsilon}(p)}\backslash\{p\}$, for small enough $r_\varepsilon=\varepsilon^s>0$, $0<s<1$, satisfying the following:
$$\hat{g}=\mathcal{A}_\varepsilon(R,a,\phi)^{\frac{8}{n-4}}g,$$
with $\sigma_2(A_{\hat{g}})=n(n-1)/8$, where

\begin{equation*}\label{eq64}
\mathcal{A}_\varepsilon(R,a,\phi)=u_{\varepsilon,R,a}+ r_\varepsilon^{-\overline\gamma}|x|^{\overline\gamma}v_\phi+h+U_{\varepsilon,R,a,\phi},
\end{equation*}
in conformal normal coordinates centered at $p$, and with
\begin{enumerate}
\item[I1)] $R^{\frac{4-n}{4}}=2(1+b)\varepsilon^{\frac{4-n}{4}}$ with $|b|\leq 1/2$;
\item[I2)]  $\phi\in C^{2,\alpha} (\mathbb{S}^{n-1}_{r_\varepsilon})^\perp$ with $\|\phi\|_{(2,\alpha),r_\varepsilon}\leq \kappa r_\varepsilon^{2+l-\delta_1}$, $l>0$ and $\delta_1>0$ is small and $\kappa>0$ is some constant to be chosen later;
\item[I3)] $|a|r_\varepsilon^{1-\delta_2}\leq 1$ is small with $3\delta_2>\max\{\delta_1,l\}$;
\item[I4)] $h=\dfrac{1}{2}\left((1-\overline\gamma)r_\varepsilon^{-\overline\gamma-1}|x|^{\overline\gamma+1}+(\overline\gamma+1)r_\varepsilon^{-\overline\gamma+1}|x|^{\overline\gamma-1}\right)\overline f$, with $\overline f=O(|x|^2)$.
\item[I5)] $U_{\varepsilon,R,a,\phi}\in C^{2,\alpha}_{(\gamma,\overline\gamma)}(B_{r}(0)\backslash\{0\})$ with $\pi''_{r_\varepsilon}(U_{\varepsilon,R,a,\phi}|_{\partial B_{r_\varepsilon}(0)})=0$, $\gamma=\delta_{n,2}+1-n/4-\varepsilon_1$, $\overline\gamma=n/4+1+\varepsilon_1$, $\varepsilon_1>0$ small, and satisfies the inequalities (\ref{eq71}) and (\ref{eq82}).
\end{enumerate}

Also, from Theorem \ref{teo02} there exists a family of constant $\sigma_2-$curvature metrics in $M_{r_\varepsilon}=M\backslash B_{r_\varepsilon}(p)$, given by (\ref{eq021}), for small enough $r_\varepsilon>0$, satisfying the following:
$$\tilde{g}=\mathcal{B}_{r_\varepsilon}(\lambda,\varphi)^ {\frac{8}{n-4}}g,$$
with $\sigma_2(A_{\tilde{g}})=n(n-1)/8$, where
\begin{equation*}\label{eq65}
\mathcal{B}_{r_\varepsilon}(\Lambda,\varphi)=f+\Lambda f G_p+fu_\varphi+ fV_{\lambda,\varphi},
\end{equation*}
 in conformal normal coordinates centered at $p$, with
\begin{enumerate}
\item[E1)] $f=1+\overline{f}$ with $\overline{f}=O(|x|^2)$;
\item[E2)] $\Lambda\in\mathbb{R}$ with $|\Lambda|^2\leq r_\varepsilon^{n+l+\delta_5}$, with $l>\delta_5$;
\item[E3)] $\varphi\in C^{2,\alpha}(\mathbb{S}^{n-1}_{r})$ is $L^2-$orthogonal to the constant functions and belongs to the ball of radius $\beta r_\varepsilon^{2+l-\delta_4}$ with $\beta$ a positive constant to be chosen later and $\delta_4<l$;
\item[E4)] $V_{\Lambda,\varphi}\in C^{2,\alpha}_\nu(M_{r_\varepsilon})$ is constant on $\partial M_{r_\varepsilon}$, satisfies the inequality (\ref{eq72}) and (\ref{eq86}).
\end{enumerate}

We want to show that there are parameters, $R\in\mathbb{R}_+$, $a\in\mathbb{R}^n$, $\Lambda\in\mathbb{R}$ and functions $\varphi,\phi\in C^{2,\alpha}(\mathbb{S}^{n-1}_{r_\varepsilon})$ such that
\begin{equation}\label{eq47}
\left\{\begin{array}{lcl}
\mathcal{A}_\varepsilon(R,a,\phi) & = & \mathcal{B}_{r_\varepsilon}(\Lambda,\varphi)\\
\partial_r\mathcal{A}_\varepsilon(R,a,\phi) & = & \partial_r\mathcal{B}_{r_\varepsilon}(\Lambda,\varphi)
\end{array}\right.
\end{equation}
on $\partial B_{r_\varepsilon}(p)$. 

First, we take $\overline f$ in I4) equal to $\overline f$ from $E1)$, $\delta_1$ in $I2)$ equal to $\delta_4$ in $E3$. Now, if we take $\omega$ and $\vartheta$ in the ball of radius $r_\varepsilon^{2+l-\delta_1}$ in $C^{2,\alpha}(\mathbb{S}^{n-1}_{r_\varepsilon})$, with $\omega$ belonging to the space spanned by the coordinate functions, $\vartheta$ belonging to the high frequencies space, and we define $\varphi:=\omega+\vartheta$, then we can apply Theorem \ref{teo02} to define the function $\mathcal{B}_{r_\varepsilon}( \Lambda,\omega+\vartheta)$, since $\|\varphi\|_{(2,\alpha),r_\varepsilon}\leq 2r_\varepsilon^{2+l-\delta_1}$.

Note that, by the definition of the function $h$ in I4), we obtain
$$\pi''_{r_\varepsilon}(\mathcal A_\varepsilon(R,a,\phi))=\phi+\pi''_{r_\varepsilon}(u_{\varepsilon,R,a}+\overline f)$$
and by the definition of $u_\varphi$ and $G_p$ in Section \ref{sec10}, we obtain that
$$\pi''_{r_\varepsilon}(\mathcal B_r(\Lambda,\varphi))=\vartheta+\pi''_{r_\varepsilon}(\overline f+\Lambda\overline f|x|^{2-\frac{n}{2}}+\overline fu_{\varphi}+\overline fV_{\Lambda,\varphi}).$$
where we are using that $\pi''_{r_\varepsilon}(u_{\varphi}|_ {\mathbb{S}^{n-1}_{r_\varepsilon}})=\vartheta$, $\pi''_{r_\varepsilon}(V_{\Lambda,\varphi} |_{\mathbb{S}^{n-1}_{r_\varepsilon}})=0$ and $f=1+\overline{f}$. 

Now define
\begin{equation}\label{eq66}
\begin{array}{lcl}
\phi_\vartheta & := & \pi''_{r_\varepsilon}((\mathcal{B}_{r_\varepsilon}( \Lambda,\omega+\vartheta)-u_{\varepsilon,R,a}-\overline f) |_{\mathbb{S}_{r_\varepsilon}^{n-1}})\\
\\
& = & \pi''_{r_\varepsilon}((\Lambda \overline f|x|^{2-\frac{n}{2}} +\overline{f}u_{\omega+\vartheta}+\overline{f} V_{\Lambda,\omega+\vartheta}- u_{\varepsilon,R,a} )|_{\mathbb{S}_{r_\varepsilon}^{n-1}})+\vartheta.
\end{array}
\end{equation}

Lets derive an estimate for $\|\phi_{\vartheta}\|_{(2,\alpha),r_\varepsilon}$.

Note that, from (\ref{eq93}), we get
\begin{equation}\label{eq92}
\pi''_{r_\varepsilon}(u_{\varepsilon,R,a} |_{\mathbb{S}_{r_\varepsilon}^{n-1}})=O(|a|^2r_\varepsilon^2),
\end{equation}
since $r_\varepsilon=\varepsilon^s$, with $s\in(0,1)$ small enough and $R^{\frac{4-n}{4}}=2(1+b)\varepsilon^{\frac{4-n}{4}}$ with $|b|\leq 1/2$, implies that $R<r_\varepsilon$.
Lets consider $a\in\mathbb{R}^n$ with $|a|^2\leq r_\varepsilon^{l}$. Hence we have that $|a|r_\varepsilon^{1-\delta_3}\leq r_\varepsilon^{1+\frac{l}{2}-\delta_3}$ tends to zero when $\varepsilon$ goes to zero, and I3) is satisfied for $\varepsilon>0$ small enough. Furthermore, since $|a|^2r_\varepsilon^2\leq r_\varepsilon^{2+l}$, we can show that
\begin{equation}\label{eq90}
\|\pi''_{r_\varepsilon}(u_{\varepsilon,R,a} |_{\mathbb{S}_{r_\varepsilon}^{n-1}})\|_{(2,\alpha),r_\varepsilon}\leq Cr_\varepsilon^{2+l},
\end{equation}
for some constant $C>0$ independent of $\varepsilon$, $R$ and $a$.

Observe that, E2) implies
\begin{equation}\label{eq91}
\|\pi''_{r_\varepsilon}(\Lambda \overline f|x|^{2-\frac{n}{2}})|_{\mathbb{S}^{n-1}_{r_\varepsilon}})\|_{(2,\alpha),r_\varepsilon}\leq Cr_\varepsilon^{2+p}.
\end{equation}
Now, using (\ref{eq031}), (\ref{eq72}), (\ref{eq66}) and the fact that $\overline{f}=O(|x|^2)$, we deduce that
\begin{equation}\label{eq99}
\|\phi_\vartheta-\vartheta \|_{(2,\alpha),r_\varepsilon}\leq cr_\varepsilon^{2+p},
\end{equation}
and
$$\|\phi_\vartheta\|_{(2,\alpha),r_\varepsilon}\leq cr_\varepsilon^{2+p-\delta_1},$$
for every $\vartheta\in \pi''(C^{2,\alpha}(\mathbb{S}^{n-1}_{r_\varepsilon}))$ in the ball of radius $r_\varepsilon^{2+p-\delta_1}$ and for some constant $c>0$ that does not depend on $\varepsilon$.  Therefore we can apply Theorem \ref{teo01} with $\kappa$ equal to this constant $c$ and $\mathcal{A}_\varepsilon(R,a,\phi_\vartheta)$ is well defined. The definition (\ref{eq66}) immediately yields
$$\pi''_{r_\varepsilon}(\mathcal{A}_\varepsilon(R,a,\phi_\vartheta) |_{\mathbb{S}_{r_\varepsilon}^{n-1}}) =\pi''_{r_\varepsilon}(\mathcal{B}_{r_\varepsilon} (\Lambda,\omega+\vartheta) |_{\mathbb{S}_{r_\varepsilon}^{n-1}}).$$

We project the second equation of the system (\ref{eq47}) on the high  frequencies space, the space of functions which are $L^2(\mathbb{S}^{n-1})-$orthogonal to $e_0$ $,\ldots,e_n$. This yields a nonlinear equation which can be written as
\begin{equation}\label{eq89}
r_\varepsilon \partial_r(v_{\vartheta}- u_{\vartheta})+ \mathcal{S}_\varepsilon(a,b,\Lambda,\omega,\vartheta)=0,
\end{equation}
on $\partial_rB_{r_\varepsilon}(0)$, where
\begin{equation}\label{eq061}
\begin{array}{c}
 \mathcal{S}_\varepsilon(a,b,\Lambda,\omega,\vartheta) = r_\varepsilon\partial_r v_{\phi_\vartheta-\vartheta}+   r_\varepsilon\partial_r \pi''_{r_\varepsilon}(u_{\varepsilon,R,a}|_{\mathbb{S}^{n-1}_{r_\varepsilon}}) \\
 + r_\varepsilon \partial_r \pi''_{r_\varepsilon}((U_{\varepsilon,R,a,\phi_\vartheta} -\Lambda\overline fG_p-\overline{f}u_{\omega+\vartheta})|_{ \mathbb{S}^{n-1}_{r_\varepsilon}})
- r_\varepsilon \partial_r \pi''_{r_\varepsilon}((\overline fV_{\Lambda,\omega+\vartheta}) |_{ \mathbb{S}^{n-1}_{r_\varepsilon}}).
\end{array}
\end{equation}

Note that here and in (\ref{eq66}) we are using the term $h$ in $\mathcal A_\varepsilon(R,a,\cdot)$ to cancel the terms $r_\varepsilon\partial_r\overline f$ and $\overline f$, respectively, which does not have the right decay, see (\ref{eq66}) and (\ref{eq061}). Also the definition of $u_\varphi$. in Section \ref{sec10} is important because we have the term $r^{-\overline\gamma}|x|^{\overline\gamma}v_\phi$ in $\mathcal A_\varepsilon(R,a,\cdot)$.

The map $\mathcal Z: C^{2,\alpha}(\mathbb{S}^{n-1})^\perp\rightarrow C^{0,\alpha}(\mathbb{S}^{n-1})^\perp$ defined by 
$$\mathcal{Z}(\vartheta):= \partial_r(v_\vartheta-\mathcal{Q}_1(\vartheta)),$$
is an isomorphism (see \cite{M}, proof of Proposition 8 in \cite{MR1712628} and proof of Proposition 2.6 in \cite{MR1763040}). On the other hand, for any $\vartheta\in C^{2,\alpha}(\mathbb{S}^{n-1}_{r_\varepsilon})^\perp$ we have
$$r_\varepsilon\partial_r (v_\vartheta- \mathcal{Q}_{r_\varepsilon}(\vartheta)) (r_\varepsilon\cdot) =\mathcal Z(\overline\vartheta),$$
where $\overline\vartheta=\vartheta(r_\varepsilon\cdot)$, see \cite{MR2639545} for more details.
Therefore to solve the equation (\ref{eq89}) it is enough to show that the map $\mathcal H_\varepsilon(a,b,\Lambda,\omega,\cdot): \mathcal{D}_\varepsilon \rightarrow C^{2,\alpha}(\mathcal{S}^{n-1})^\perp$ given by
$$\mathcal H_\varepsilon(a,b,\Lambda,\omega,\vartheta)= -\mathcal{Z}^{-1}(\mathcal{S}_\varepsilon(a,b,\Lambda,\omega, \vartheta_{r_\varepsilon}) (r_\varepsilon\cdot)),$$
has a fixed point, where $\mathcal{D}_\varepsilon:=\{\vartheta\in \pi''(C^{2,\alpha}(\mathbb{S}^{n-1}));\|\vartheta\|_{(2,\alpha),1}\leq r_\varepsilon^{2+l-\delta_1}\}$ and $\vartheta_{r_\varepsilon}(x):=\vartheta(r_\varepsilon^{-1}x)$.

\begin{lemma}\label{lem08}
There is a constant $\varepsilon_0>0$ such that if $\varepsilon\in(0,\varepsilon_0)$, $a\in\mathbb{R}^n$ with $|a|^2\leq r_\varepsilon^{l}$, $b$ and $\Lambda$ in $\mathbb{R}$ with $|b|\leq 1/2$ and $|\Lambda|^2\leq r_\varepsilon^{n+l+\delta_5}$, and $\omega\in C^{2,\alpha}(\mathbb{S}^{n-1}_{r_\varepsilon})$ belongs to the space spanned by the coordinate functions and with norm bounded by $r_\varepsilon^{2+l-\delta_1}$, then the map $\mathcal H_\varepsilon(a,b,\Lambda,\omega,\cdot)$ has a fixed point in $\mathcal{D}_\varepsilon$.
\end{lemma}
\begin{proof}
First note that by (\ref{eq99}), $\phi_0$ satisfies
$$\|\phi_0\|_{(2,\alpha),r_\varepsilon}\leq cr_\varepsilon^{2+l},$$
where the constant $c>0$ is independent of $\varepsilon$.

From (\ref{eq73}), (\ref{eq71}), (\ref{eq031}),  (\ref{eq72}), (\ref{eq90}) and (\ref{eq91}) and the fact that $\overline{f}=O(|x|^2)$ we obtain
\begin{equation*}\label{eq97}
 \|\mathcal{S}_\varepsilon (a,b,\lambda,\omega,0)\|_{(1,\alpha),r_\varepsilon}\leq cr_\varepsilon^{2+l},
\end{equation*}
for some constant $c>0$ independent of $\varepsilon$. 

Now, if $\vartheta_1,\vartheta_2\in\mathcal{D}_\varepsilon$, then
$$\hspace{-4cm}\|\mathcal H_\varepsilon(a,b,\Lambda,\omega,\vartheta_1)- \mathcal H_\varepsilon(a,b,\Lambda,\omega,\vartheta_2)\|_{(2,\alpha),1}$$
$$\hspace{-3cm}\leq C\left(\|r_\varepsilon\partial_r v_{\phi_{\vartheta_{r_\varepsilon,1}}-\vartheta_{r_\varepsilon,1} -(\phi_{\vartheta_{r_\varepsilon,2}} -\vartheta_{r_\varepsilon,2})}\|_{(1,\alpha),r_\varepsilon}\right.$$
$$\hspace{-1cm}+ \|r_\varepsilon \partial_r\pi''_{r_\varepsilon} ((U_{\varepsilon,R,a,\phi_{\vartheta_{r_\varepsilon,1}}}- U_{\varepsilon,R,a,\phi_{\vartheta_{r_\varepsilon,2}}}) |_{\mathbb{S}^{n-1}_{r_\varepsilon}}) \|_{(1,\alpha),r_\varepsilon}$$
$$+ \|r_\varepsilon \partial_r\pi''_{r_\varepsilon}( (f(V_{\Lambda,\omega+\vartheta_{r_\varepsilon,1}}- V_{\Lambda,\omega+\vartheta_{r_\varepsilon,2}})) |_{\mathbb{S}^{n-1}_{r_\varepsilon}})\|_{(1,\alpha),r_\varepsilon}$$
$$\left.+\|r_\varepsilon \partial_r \pi''_{r_\varepsilon}((\overline{f}u_{\vartheta_{r_\varepsilon,1} -\vartheta_{r_\varepsilon,2}})|_{\mathbb{S}^{n-1}_{r_\varepsilon}}) \|_{(1,\alpha),r_\varepsilon}\right),$$
where, by (\ref{eq66}) we get
$$\phi_{\vartheta_{r_\varepsilon,1}}-\vartheta_{r_\varepsilon,1} -(\phi_{\vartheta_{r_\varepsilon,2}} -\vartheta_{r_\varepsilon,2})=\pi''_{r_\varepsilon} ((\overline{f}u_{\vartheta_{r_\varepsilon,1}-\vartheta_{r_\varepsilon,2}} +\overline{f} (V_{\lambda,\omega+\vartheta_{r_\varepsilon,1}}- V_{\lambda,\omega+\vartheta_{r_\varepsilon,2}})) |_{\mathbb{S}_{r_\varepsilon}^{n-1}}).$$

Using the inequalities (\ref{eq031}), (\ref{eq86}) and the fact that $\overline{f}=O(|x|^2)$, we obtain
$$\|\phi_{\vartheta_{r_\varepsilon,1}}-\vartheta_{r_\varepsilon,1}- (\phi_{\vartheta_{r_\varepsilon,2}}-\vartheta_{r_\varepsilon,2}) \|_{(2,\alpha),r_\varepsilon}\leq cr_\varepsilon^{\delta_6} \|\vartheta_{r_\varepsilon,1}- \vartheta_{r_\varepsilon,2}\|_{(2,\alpha),r_\varepsilon},$$
for some constants $\delta_6>0$ and $c>0$ that does not depend on $\varepsilon$. Using (\ref{eq73}) we have that
\begin{equation}\label{eq108}
\|r_\varepsilon\partial_r v_{\phi_{\vartheta_{r_\varepsilon,1}}-\vartheta_{r_\varepsilon,1} -(\phi_{\vartheta_{r_\varepsilon,2}} -\vartheta_{r_\varepsilon,2})}\|_{(1,\alpha),r_\varepsilon}\leq cr_\varepsilon^{\delta_6}\|\vartheta_1- \vartheta_2\|_{(2,\alpha),1}.
\end{equation}

From (\ref{eq82}) and (\ref{eq86}) we conclude that
$$\|U_{\varepsilon,R,a,\phi_{\vartheta_{r_\varepsilon,1}}}- U_{\varepsilon,R,a,\phi_{\vartheta_{r_\varepsilon,2}}}\|_{(2,\alpha), [\frac{1}{2}r_\varepsilon,r_\varepsilon]}\leq Cr_\varepsilon^{\delta_1} \|\vartheta_{r_\varepsilon,1}- \vartheta_{r_\varepsilon,2}\|_{(2,\alpha),r_\varepsilon}$$
and
$$\|V_{\Lambda,\omega+\vartheta_{r_\varepsilon,1}}- V_{\Lambda,\omega+\vartheta_{r_\varepsilon,2}}\|_{(2,\alpha), [r_\varepsilon,2r_\varepsilon]}\leq Cr_\varepsilon^{\delta_5} \|\vartheta_{r_\varepsilon,1}- \vartheta_{r_\varepsilon,2}\|_{(2,\alpha),r_\varepsilon},$$
for some $\delta_1>0$ and $\delta_5>0$ independent of $\varepsilon$. From this, (\ref{eq031}) and the fact that $f=1+\overline{f}$, we derive an estimate as (\ref{eq108}) for the other terms, and therefore we get
\begin{equation}\label{eq98}
\|\mathcal H_\varepsilon(a,b,\Lambda,\omega,\vartheta_1)- \mathcal H_\varepsilon(a,b,\Lambda,\omega,\vartheta_2)\|_{(2,\alpha),1}\leq \frac{1}{2} \|\vartheta_1-\vartheta_2\|_{(2,\alpha),1},
\end{equation}
for all $\vartheta_1,\vartheta_2\in\mathcal{D}_\varepsilon$, since $\varepsilon>0$ is small enough. From this and (\ref{eq98}) we get the result.
\end{proof}

Therefore there exists a unique solution of (\ref{eq89}) in the ball of radius $r_\varepsilon^{2+l-\delta_1}$ in $C^{2,\alpha}(\mathbb{S}^{n-1}_{r_\varepsilon})$. We denote by $\vartheta_{\varepsilon, a,b,\Lambda,\omega}$ this solution given by Lemma \ref{lem08}. Since this  solution is obtained through the application of fixed point theorem for  contraction mappings, it is continuous with respect to the parameters $\varepsilon$, $a$, $b$, $\Lambda$ and $\omega$.

Now, recall that $R^{\frac{4-n}{4}}=2(1+b)\varepsilon^{\frac{4-n}{4}}$ with $|b|\leq 1/2$. Hence, using (\ref{eq92}) and Corollary \ref{cor02} and \ref{cor04} we show that
$$\begin{array}{rcl}
u_{\varepsilon,R,a}(r_\varepsilon\theta) & = & 1+b+\displaystyle \frac{\varepsilon^{\frac{n-4}{2}}}{4(1+b)} {r_\varepsilon}^{2-\frac{n}{2}}+ \left(\frac{n-4}{2}u_{\varepsilon,R}(r_\varepsilon\theta)+r\partial_r u_{\varepsilon,R}(r_\varepsilon\theta)\right)a\cdot x\\
& + & O(|a|^2r_\varepsilon^2) +O(\varepsilon^{\frac{n+4}{2}}{r_\varepsilon} ^{-\frac{n}{2}}),
\end{array}$$
where the last term does not depend on $\theta$. Hence, we have
$$\mathcal{A}_\varepsilon(R,a,\phi_{\vartheta_ {\varepsilon,a,b,\lambda,\omega}})(r_\varepsilon\theta) =1+b+ \displaystyle\frac{\varepsilon^{\frac{n-4}{2}}}{4(1+b)}{r_\varepsilon}^{2-\frac{n}{2}} +v_{\phi_{\vartheta_{\varepsilon, a,b,\lambda,\omega}}}(r_\varepsilon\theta)+\overline f(r_\varepsilon\theta)$$
$$+ \left(\frac{n-4}{2}u_{\varepsilon,R}(r_\varepsilon\theta)+ r_\varepsilon \partial_ru_{\varepsilon,R}(r_\varepsilon\theta)\right)r_\varepsilon a\cdot \theta $$
$$\hspace{2cm}+ U_{\varepsilon,R,a,\phi_{\vartheta_{\varepsilon, a,b,\lambda,\omega}}}(r_\varepsilon\theta)+ O(|a|^2r_\varepsilon^2)+ O(\varepsilon^{\frac{n+4}{2}}r_\varepsilon^{-\frac{n}{2}}).$$

In the exterior manifold $M_{r_\varepsilon}$, in conformal normal coordinate system in the neighborhood of $\partial M_{r_\varepsilon}$, namely $\Omega_{r_\varepsilon,\frac{1}{2}r_1}$, we have
$$\mathcal{B}_{r_\varepsilon}(\Lambda,\omega+\vartheta_{\varepsilon, a,b,\Lambda,\omega})(r_\varepsilon\theta) = 1+\Lambda r_\varepsilon^{2-\frac{n}{2}}+u_{\omega+\vartheta_{\varepsilon, a,b,\lambda,\omega}}(r_\varepsilon\theta) + \overline{f}(r_\varepsilon\theta)$$
$$+ (\overline{f}u_{\omega+\vartheta_{\varepsilon, a,b,\Lambda,\omega}})(r_\varepsilon\theta) + (fV_{\Lambda,\omega+\vartheta_{\varepsilon, a,b,\Lambda,\omega}})(r_\varepsilon\theta) + O(|\Lambda| r_\varepsilon^{4-\frac{n}{2}}).$$
We now project the system (\ref{eq47}) on the set of functions spanned by the constant function. This yields the equations
\begin{equation}\label{eq03}
\left\{\begin{array}{rcl}
b+\left( \displaystyle\frac{\varepsilon^{\frac{n-4}{2}}}{4(1+b)}- \Lambda\right)r_\varepsilon^{2-\frac{n}{2}} & = & H_{0,\varepsilon}(a,b,\Lambda,\omega)\\
\left(2-\frac{n}{2}\right)\left( \displaystyle\frac{\varepsilon^\frac{n-4}{2}}{4(1+b)}- \Lambda\right)r_\varepsilon^{2-\frac{n}{2}} & = & r_\varepsilon\partial_rH_{0,\varepsilon}(a,b,\Lambda,\omega)
\end{array}\right.,
\end{equation}
where $H_{0,\varepsilon}$ and $\partial_r H_{0,\varepsilon}$ are continuous maps and satisfy
\begin{equation}\label{eq00}
H_{0,\varepsilon}(a,b,\Lambda,\omega)= O(r_\varepsilon^{2+l})\;\;\;\mbox{ and }\;\;\;r_\varepsilon\partial_r H_{0,\varepsilon}(a,b,\Lambda,\omega)= O(r_\varepsilon^{2+l}).
\end{equation}
\begin{lemma}\label{lem12}
There is a constant $\varepsilon_2>0$ such that if $\varepsilon\in(0,\varepsilon_2)$, $a\in\mathbb{R}^n$ with $|a|^2\leq r_\varepsilon^{l}$ and $\omega\in C^{2,\alpha}(\mathbb{S}^{n-1}_{r_\varepsilon})$ belongs to the space spanned by the coordinate functions and has norm bounded by $r_\varepsilon^{2+l-\delta_1}$, then the system (\ref{eq03}) has a solution $(b,\Lambda)\in\mathbb{R}^2$, with $|b|\leq 1/2$ and $|\Lambda|^2\leq r_\varepsilon^{n+l+\delta_5}$.
\end{lemma}
\begin{proof} Define a continuous map $\mathcal{G}_{\varepsilon,a,\omega}: \mathcal{D}_{0,\varepsilon}\rightarrow \mathbb{R}^2$ by
$$\begin{array}{rcl}
\mathcal{G}_{\varepsilon,a,\omega}(b,\Lambda) & := & \displaystyle \left(\frac{2r_\varepsilon}{n-4}\partial_rH_{0,\varepsilon} (a,b,\Lambda,\omega)+H_{0,\varepsilon} (a,b,\Lambda,\omega),\right.\\
\\
& & \displaystyle \left.\frac{\varepsilon^{\frac{n-4}{2}}}{4(1+b)}+\frac{2r_\varepsilon^{\frac{n}{2}-1}}{n-4} \partial_rH_{0,\varepsilon} (a,b,\Lambda,\omega)\right),
\end{array}$$
where $\mathcal{D}_{0,\varepsilon}:= \{(b,\Lambda)\in\mathbb{R}^2; |b|\leq 1/2\mbox{ and } |\Lambda|\leq r_\varepsilon^{\frac{n+l+\delta_5}{2}}\}$.

Then, using (\ref{eq00}) and $r_\varepsilon=\varepsilon^s$, with $s\in(0,1)$ small enough, we can show that
$\mathcal{G}_{\varepsilon,a,\omega}(\mathcal{D}_{0,\varepsilon})\subset \mathcal{D}_{0,\varepsilon},$ for small enough $\varepsilon>0$. With the estimate that we have above we can show that this map is a contraction, and hence it has a fixed point which depends continuously on parameter $\varepsilon$, $a$ and $\omega$. Obviously, this fixed point is a solution of the system (\ref{eq03}).
\end{proof}

From now on we will work with the fixed point given by Lemma \ref{lem12} and we will write simply as $(b,\Lambda)$.

Finally, we project the system (\ref{eq47}) over the space of functions spanned by the coordinate functions. It will be convenient to decompose $\omega$ in
\begin{equation}\label{eq102}
\omega=\sum_{i=1}^n\omega_ie_i,\;\;\;\mbox{ where }\;\;\;\omega_i=\int_{\mathbb{S}^{n-1}}\omega(r_\varepsilon\cdot)e_i.
\end{equation}

Hence, $|\omega_i|\leq c_n\sup_{\mathbb{S}^{n-1}_{r_\varepsilon}}|\omega|.$ From this and Proposition \ref{propo002} we get the system
\begin{equation}\label{eq101}
\left\{\begin{array}{rcl}
F(r_\varepsilon)r_\varepsilon a_i-\omega_i & = & H_{i,\varepsilon}(a,\omega)\\
G(r_\varepsilon)r_\varepsilon a_i-(1-n)\omega_i & = & r_\varepsilon\partial_r H_{i,\varepsilon}(a,\omega),
\end{array}\right.
\end{equation}
$i=1,\ldots,n$, where
$$\begin{array}{rcl}
  F(r_\varepsilon) & := & \displaystyle\frac{n-4}{2}
u_{\varepsilon,R}(r_\varepsilon\theta)+ r_\varepsilon\partial_r u_{\varepsilon,R}(r_\varepsilon\theta),\\
\\
G(r_\varepsilon) & := & \displaystyle\frac{n-4}{2}u_{\varepsilon,R}(r_\varepsilon\theta)+ \frac{n}{2} r_\varepsilon\partial_r u_{\varepsilon,R}(r_\varepsilon\theta) +r_\varepsilon^2\partial_r^2 u_{\varepsilon,R}(r_\varepsilon\theta),
  \end{array}
$$
\begin{equation}\label{eq103}
H_{i,\varepsilon}(a,\omega)= O(r_\varepsilon^{2+l})
\;\;\;\mbox{ and }\;\;\;r_\varepsilon\partial_rH_{i,\varepsilon}(a,\omega)= O(r_\varepsilon^{2+l}).
\end{equation}
The maps $H_{i,\varepsilon}$ and $\partial_rH_{i,\varepsilon}$ are continuous.

\begin{lemma}\label{lem13}
There is a constant $\varepsilon_2>0$ such that if $\varepsilon\in(0,\varepsilon_2)$ then the system (\ref{eq101}) has a solution $(a,\omega)\in\mathbb{R}^n\times C^{2,\alpha}(\mathbb{S}^{n-1}_{r_\varepsilon})$ with $|a|^2\leq r_\varepsilon^{l}$ and $\omega$ given by (\ref{eq102}) of norm bounded by $r_{\varepsilon}^{2+l-\delta_1}$.
\end{lemma}
\begin{proof}
Define a continuous map $\mathcal{K}_{i,\varepsilon}:\mathcal{D}_{i,\varepsilon}\rightarrow \mathbb{R}^2$ by
$$\mathcal{K}_{i,\varepsilon}(a_i,\omega_i) := \left( (G(r_\varepsilon)+ (n-1)F(r_\varepsilon))^{-1}r_\varepsilon^{-1} (r_\varepsilon\partial_rH_{i,\varepsilon}(a,\omega)+ (n-1)H_{i,\varepsilon}),\right.$$
$$
\hspace{1,5cm}\left. (G(r_\varepsilon)+ (n-1)F(r_\varepsilon))^{-1}F(r_\varepsilon) (r_\varepsilon\partial_rH_{i,\varepsilon}(a,\omega) +(n-1)H_{i,\varepsilon})-H_{i,\varepsilon}\right),$$
where $\mathcal{D}_{i,\varepsilon}:= \{(a_i,\omega_i)\in\mathbb{R}^2;|a_i|^2\leq n^{-1}r_\varepsilon^{l}\mbox{ and }|\omega_i|\leq n^{-1}k_{i,n}^{-1}r_\varepsilon^{2+l-\delta_1}\}$, $k_{i,n}=\|e_i\|_{(2,\alpha),1}$, $F(r_\varepsilon) = (n-2)(1+b)+ O(\varepsilon^{2-s(n-2)})$ and $G(r_\varepsilon)+ (n-1)F(r_\varepsilon)= n(n-2)(1+b)+ O(\varepsilon^{2-s(n-2)})$ with $2-s(n-2)>0$. 

From (\ref{eq103}) we obtain that $\mathcal{K}_{i,\varepsilon}(\mathcal{D}_{i,\varepsilon})\subset \mathcal{D}_{i,\varepsilon},$ for small enough $\varepsilon>0$. By the Brouwer's fixed point theorem there exists a fixed point of the map $\mathcal{K}_{i,\varepsilon}$ and this fixed point is a solution of the system (\ref{eq101}).
\end{proof}

Now we are ready to prove the main theorem of this paper.
\begin{theorem}
 Let $(M^n,g_0)$ be an compact Riemannian manifold nondegenerate with dimension $n\geq 5$, $g_0$ conformal to some $2-$admissible metric and the $\sigma_2-$curvature is equal to $\frac{n(n-1)}{8}$. Let ${\{p_1,\ldots,p_m\}}$ a set of points in $M$ such that $\nabla_{g_0}^jW_{g_0}(p_i)=0$ for $j=0,\ldots,\left[\frac{n-4}{2}\right]$ and $i=1,\ldots,m$, where $W_{g_0}$ is the Weyl tensor of the metric $g_0$. Then, there exist a constant $\varepsilon_0>0$ and a one-parameter family of complete metrics $g_\varepsilon$ on $M\backslash\{p_1,\ldots,p_m\}$ defined for $\varepsilon\in(0,\varepsilon_0)$ such that
 \begin{enumerate}
  \item each $g_\varepsilon$ is conformal to $g_0$ and has constant $\sigma_2-$curvature $\sigma_2(A_{g_\varepsilon})=\frac{n(n-1)}{8}$;
  \item $g_\varepsilon$ is asymptotically Delaunay near each point $p_i$, for all $i=1,\ldots,m$;
  \item $g_\varepsilon\rightarrow g_0$ uniformly on compact sets in $M\backslash\{p_1,\ldots,p_m\}$ as $\varepsilon\rightarrow 0$.
 \end{enumerate}
\end{theorem}
\begin{proof} First we proof the theorem when $m=1$ and then we will explain the minor changes that need to be made in order to deal with more than one singular point.

We keep the previous notations.

From Lemmas \ref{lem08}, \ref{lem12} and \ref{lem13} we conclude that there is $\varepsilon_0>0$ such that for all $\varepsilon\in(0,\varepsilon_0)$ there are parameters $R_\varepsilon$, $a_\varepsilon$, $\phi_\varepsilon$, $\Lambda_\varepsilon$ and $\varphi_\varepsilon$ for which the functions $\mathcal{A}_\varepsilon (R_\varepsilon,a_\varepsilon,\phi_\varepsilon)$ and $\mathcal{B}_{r_\varepsilon}(\Lambda_\varepsilon,\varphi_\varepsilon)$ coincide up to order one in $\partial B_{r_\varepsilon}(p)$. Since $g_0$ is conformal to some $2-$admissible metric, we can use elliptic regularity to show that the function $\mathcal{U}_\varepsilon$ defined by $\mathcal{U}_\varepsilon:=\mathcal{A}_\varepsilon (R_\varepsilon,a_\varepsilon,\phi_\varepsilon)$ in $B_{r_\varepsilon}(p)\backslash\{p\}$ and $\mathcal{U}_\varepsilon:= \mathcal{B}_{r_\varepsilon}(\Lambda_\varepsilon,\varphi_\varepsilon)$ in $M\backslash B_{r_\varepsilon}(p)$ is a positive smooth function in $M\backslash\{p\}$. Moreover, since $\mathcal{A}_\varepsilon (R_\varepsilon,a_\varepsilon,\phi_\varepsilon)\geq c|x|^{\frac{4-n}{4}}$, for some constant $c>0$, then the function $\mathcal{U}_\varepsilon$ tends to infinity on approach to $p$ with sufficiently fast rate.

Therefore, $g_\varepsilon:=\mathcal{U}_\varepsilon^{\frac{8}{n-4}}g$ is a one-parameter family of complete smooth metric defined in $M\backslash\{p\}$ and by Theorem \ref{teo01} and \ref{teo02} it satisfies i), ii) and iii).

To proof the general case, we will just explain the minor changes that need to be made. We direct the reader to \cite{MR2639545} for more details.

The interior analysis is done around at each point $p_i$, where we can find a family of metrics defined in $B_{r_{\varepsilon_i}}(p)\backslash\{p\}$, with $\varepsilon_i=t_i\varepsilon$, $\varepsilon>0$, $t_i\in(\delta,\delta^{-1})$ and $\delta>0$ fixed, $i=1,\ldots,m$.

The exterior analysis is done considering some changes. First we consider conformal normal coordinates around at point $p_i$. Using this we define the spaces $C^{l,\alpha}_\nu(M\backslash\{p_1,\ldots,p_m\})$ and $C^{2,\alpha}_\nu(M_r)$ like in the Section 6 in \cite{MR2639545}. For each $\varphi=(\varphi_1,\ldots,\varphi_m)$ with $\varphi_i\in C^{2,\alpha}(\mathbb S_r^{n-1})$ a function $L^2-$orthogonal to the constant functions, we define $u_\varphi\in C^{2,\alpha}_\nu(M_r)$ such that near each point $p_i$, the function $u_\varphi$ is like in the Section \ref{sec10}. Then we proof an analog of Theorem \ref{teo03} in this context, with $w\in\mathbb R$ constant on any component of $\partial M_r$. 

Therefore, using again that $g_0$ is conformal to some $2-$admissible metric, we can use elliptic regularity to get the result.
\end{proof}


\begin{bibdiv}
  \begin{biblist}
\bib{MR0431287}{article}{
   author={Aubin, T.},
   title={\'Equations diff\'erentielles non lin\'eaires et probl\`eme de
   Yamabe concernant la courbure scalaire},
   journal={J. Math. Pures Appl. (9)},
   volume={55},
   date={1976},
   number={3},
   pages={269--296},
   issn={0021-7824},
   review={\MR{0431287 (55 \#4288)}},
} 

\bib{BPS}{article}{
   author={Bettiol, R. G.},
   author={Piccione, P.},
   author={Santoro, Bianca},
   title={Bifurcation of periodic solutions to the singular Yamabe problem on spheres},
   journal={To appear in J. Differential Geom.},
}

\bib{MR2389992}{article}{
   author={Branson, T. P.},
   author={Gover, A. R.},
   title={Variational status of a class of fully nonlinear curvature
   prescription problems},
   journal={Calc. Var. Partial Differential Equations},
   volume={32},
   date={2008},
   number={2},
   pages={253--262},
   issn={0944-2669},
   review={\MR{2389992 (2008m:53080)}},
   doi={10.1007/s00526-007-0141-6},
}

\bib{MR982351}{article}{
   author={Caffarelli, L. A.},
   author={Gidas, B.},
   author={Spruck, J.},
   title={Asymptotic symmetry and local behavior of semilinear elliptic
   equations with critical Sobolev growth},
   journal={Comm. Pure Appl. Math.},
   volume={42},
   date={1989},
   number={3},
   pages={271--297},
   issn={0010-3640},
   review={\MR{982351 (90c:35075)}},
   doi={10.1002/cpa.3160420304},
}

\bib{MR3023858}{article}{
   author={Catino, G.},
   author={Mazzieri, L.},
   title={Connected sum construction for $\sigma\sb k$-Yamabe metrics},
   journal={J. Geom. Anal.},
   volume={23},
   date={2013},
   number={2},
   pages={812--854},
   issn={1050-6926},
   review={\MR{3023858}},
   doi={10.1007/s12220-011-9265-1},
}

\bib{MR1945280}{article}{
   author={Chang, S.-Y. A.},
   author={Gursky, M. J.},
   author={Yang, P.},
   title={An a priori estimate for a fully nonlinear equation on
   four-manifolds},
   note={Dedicated to the memory of Thomas H.\ Wolff},
   journal={J. Anal. Math.},
   volume={87},
   date={2002},
   pages={151--186},
   issn={0021-7670},
   review={\MR{1945280 (2003k:53036)}},
   doi={10.1007/BF02868472},
}

\bib{MR1923964}{article}{
   author={Chang, S.-Y. A.},
   author={Gursky, M. J.},
   author={Yang, P. C.},
   title={An equation of Monge-Amp\`ere type in conformal geometry, and
   four-manifolds of positive Ricci curvature},
   journal={Ann. of Math. (2)},
   volume={155},
   date={2002},
   number={3},
   pages={709--787},
   issn={0003-486X},
   review={\MR{1923964 (2003j:53048)}},
   doi={10.2307/3062131},
}

\bib{MR2165306}{article}{
   author={Chang, S.-Y. A.},
   author={Han, Z.-C.},
   author={Yang, P.},
   title={Classification of singular radial solutions to the $\sigma\sb k$
   Yamabe equation on annular domains},
   journal={J. Differential Equations},
   volume={216},
   date={2005},
   number={2},
   pages={482--501},
   issn={0022-0396},
   review={\MR{2165306 (2006d:53033)}},
   doi={10.1016/j.jde.2005.05.005},
}

\bib{MR2040327}{article}{
   author={Chang, S.-Y. A.},
   author={Hang, F.},
   author={Yang, P. C.},
   title={On a class of locally conformally flat manifolds},
   journal={Int. Math. Res. Not.},
   date={2004},
   number={4},
   pages={185--209},
   issn={1073-7928},
   review={\MR{2040327 (2005d:53051)}},
   doi={10.1155/S1073792804132133},
}

\bib{MR2274812}{book}{
   author={Chow, B.},
   author={Lu, P.},
   author={Ni, L.},
   title={Hamilton's Ricci flow},
   series={Graduate Studies in Mathematics},
   volume={77},
   publisher={American Mathematical Society, Providence, RI; Science Press,
   New York},
   date={2006},
   pages={xxxvi+608},
   isbn={978-0-8218-4231-7},
   isbn={0-8218-4231-5},
   review={\MR{2274812 (2008a:53068)}},
   doi={10.1090/gsm/077},
}

\bib{MR2169873}{article}{
   author={Gonz{\'a}lez, M. d. M.},
   title={Singular sets of a class of locally conformally flat manifolds},
   journal={Duke Math. J.},
   volume={129},
   date={2005},
   number={3},
   pages={551--572},
   issn={0012-7094},
   review={\MR{2169873 (2006d:53034)}},
   doi={10.1215/S0012-7094-05-12934-9},
}

\bib{MR3237067}{article}{
   author={Gonzalez, M. d. M.},
   author={Mazzieri, L.},
   title={Singularities for a fully non-linear elliptic equation in
   conformal geometry},
   journal={Bull. Inst. Math. Acad. Sin. (N.S.)},
   volume={9},
   date={2014},
   number={2},
   pages={223--244},
   issn={2304-7909},
   review={\MR{3237067}},
}

\bib{MR1938739}{article}{
   author={Guan, P.},
   author={Viaclovsky, J.},
   author={Wang, G.},
   title={Some properties of the Schouten tensor and applications to
   conformal geometry},
   journal={Trans. Amer. Math. Soc.},
   volume={355},
   date={2003},
   number={3},
   pages={925--933 (electronic)},
   issn={0002-9947},
   review={\MR{1938739 (2003h:53054)}},
   doi={10.1090/S0002-9947-02-03132-X},
}

\bib{MR1978409}{article}{
   author={Guan, P.},
   author={Wang, G.},
   title={A fully nonlinear conformal flow on locally conformally flat
   manifolds},
   journal={J. Reine Angew. Math.},
   volume={557},
   date={2003},
   pages={219--238},
   issn={0075-4102},
   review={\MR{1978409 (2004e:53101)}},
   doi={10.1515/crll.2003.033},
}

\bib{MR2373147}{article}{
   author={Gursky, M. J.},
   author={Viaclovsky, J. A.},
   title={Prescribing symmetric functions of the eigenvalues of the Ricci
   tensor},
   journal={Ann. of Math. (2)},
   volume={166},
   date={2007},
   number={2},
   pages={475--531},
   issn={0003-486X},
   review={\MR{2373147 (2008k:53068)}},
   doi={10.4007/annals.2007.166.475},
}

\bib{MR2737708}{article}{
   author={Han, Z.-C.},
   author={Li, Y. Y.},
   author={Teixeira, E. V.},
   title={Asymptotic behavior of solutions to the $\sigma\sb k$-Yamabe
   equation near isolated singularities},
   journal={Invent. Math.},
   volume={182},
   date={2010},
   number={3},
   pages={635--684},
   issn={0020-9910},
   review={\MR{2737708 (2011i:53045)}},
   doi={10.1007/s00222-010-0274-7},
}

\bib{M}{article}{
   author={Jleli, M.},
   title={Constant mean curvature hypersurfaces},
   journal={PhD Thesis, University of Paris 12},
   date={2004},
}

\bib{MR2194146}{article}{
   author={Jleli, M.},
   author={Pacard, F.},
   title={An end-to-end construction for compact constant mean curvature
   surfaces},
   journal={Pacific J. Math.},
   volume={221},
   date={2005},
   number={1},
   pages={81--108},
   issn={0030-8730},
   review={\MR{2194146 (2007a:53011)}},
   doi={10.2140/pjm.2005.221.81},
}

\bib{MR2337310}{article}{
   author={Kaabachi, S.},
   author={Pacard, F.},
   title={Riemann minimal surfaces in higher dimensions},
   journal={J. Inst. Math. Jussieu},
   volume={6},
   date={2007},
   number={4},
   pages={613--637},
   issn={1474-7480},
   review={\MR{2337310 (2008f:53006)}},
   doi={10.1017/S1474748007000060},
}

\bib{MR2477893}{article}{
   author={Khuri, M. A.},
   author={Marques, F. C.},
   author={Schoen, R. M.},
   title={A compactness theorem for the Yamabe problem},
   journal={J. Differential Geom.},
   volume={81},
   date={2009},
   number={1},
   pages={143--196},
   issn={0022-040X},
   review={\MR{2477893 (2010e:53065)}},
}

\bib{MR1666838}{article}{
   author={Korevaar, N.},
   author={Mazzeo, R.},
   author={Pacard, F.},
   author={Schoen, R.},
   title={Refined asymptotics for constant scalar curvature metrics with
   isolated singularities},
   journal={Invent. Math.},
   volume={135},
   date={1999},
   number={2},
   pages={233--272},
   issn={0020-9910},
   review={\MR{1666838 (2001a:35055)}},
   doi={10.1007/s002220050285},
}

\bib{MR888880}{article}{
   author={Lee, J. M.},
   author={Parker, T. H.},
   title={The Yamabe problem},
   journal={Bull. Amer. Math. Soc. (N.S.)},
   volume={17},
   date={1987},
   number={1},
   pages={37--91},
   issn={0273-0979},
   review={\MR{888880 (88f:53001)}},
   doi={10.1090/S0273-0979-1987-15514-5},
}

\bib{MR1988895}{article}{
   author={Li, A.},
   author={Li, Y. Y.},
   title={On some conformally invariant fully nonlinear equations},
   journal={Comm. Pure Appl. Math.},
   volume={56},
   date={2003},
   number={10},
   pages={1416--1464},
   issn={0010-3640},
   review={\MR{1988895 (2004e:35072)}},
   doi={10.1002/cpa.10099},
}

\bib{MR2393072}{article}{
   author={Marques, F. C.},
   title={Isolated singularities of solutions to the Yamabe equation},
   journal={Calc. Var. Partial Differential Equations},
   volume={32},
   date={2008},
   number={3},
   pages={349--371},
   issn={0944-2669},
   review={\MR{2393072 (2010b:35134)}},
   doi={10.1007/s00526-007-0144-3},
}

\bib{MR1712628}{article}{
   author={Mazzeo, R.},
   author={Pacard, F.},
   title={Constant scalar curvature metrics with isolated singularities},
   journal={Duke Math. J.},
   volume={99},
   date={1999},
   number={3},
   pages={353--418},
   issn={0012-7094},
   review={\MR{1712628 (2000g:53035)}},
   doi={10.1215/S0012-7094-99-09913-1},
}

\bib{MR1399537}{article}{
   author={Mazzeo, R.},
   author={Pollack, D.},
   author={Uhlenbeck, K.},
   title={Connected sum constructions for constant scalar curvature metrics},
   journal={Topol. Methods Nonlinear Anal.},
   volume={6},
   date={1995},
   number={2},
   pages={207--233},
   issn={1230-3429},
   review={\MR{1399537 (97e:53076)}},
}

\bib{MN}{article}{
   author={Mazzieri, L.},
   author={Ndiaye, C. B.},
   title={Existence of solutions for the singular $\sigma_k-$Yamabe problem},
   pages={preprint}
}

\bib{MR2900438}{article}{
   author={Mazzieri, L.},
   author={Segatti, A.},
   title={Constant $\sigma\sb k$-curvature metrics with Delaunay type ends},
   journal={Adv. Math.},
   volume={229},
   date={2012},
   number={6},
   pages={3147--3191},
   issn={0001-8708},
   review={\MR{2900438}},
   doi={10.1016/j.aim.2012.02.007},
}

\bib{MR1763040}{book}{
   author={Pacard, F.},
   author={Rivi{\`e}re, T.},
   title={Linear and nonlinear aspects of vortices},
   series={Progress in Nonlinear Differential Equations and their
   Applications, 39},
   note={The Ginzburg-Landau model},
   publisher={Birkh\"auser Boston, Inc., Boston, MA},
   date={2000},
   pages={x+342},
   isbn={0-8176-4133-5},
   review={\MR{1763040 (2001k:35066)}},
   doi={10.1007/978-1-4612-1386-4},
}

\bib{MR788292}{article}{
   author={Schoen, R.},
   title={Conformal deformation of a Riemannian metric to constant scalar
   curvature},
   journal={J. Differential Geom.},
   volume={20},
   date={1984},
   number={2},
   pages={479--495},
   issn={0022-040X},
   review={\MR{788292 (86i:58137)}},
}

\bib{MR2362323}{article}{
   author={Sheng, W.-M.},
   author={Trudinger, N. S.},
   author={Wang, X.-J.},
   title={The Yamabe problem for higher order curvatures},
   journal={J. Differential Geom.},
   volume={77},
   date={2007},
   number={3},
   pages={515--553},
   issn={0022-040X},
   review={\MR{2362323 (2008i:53048)}},
}

\bib{MR2639545}{article}{
   author={Silva Santos, A.},
   title={A construction of constant scalar curvature manifolds with
   Delaunay-type ends},
   journal={Ann. Henri Poincar\'e},
   volume={10},
   date={2010},
   number={8},
   pages={1487--1535},
   issn={1424-0637},
   review={\MR{2639545 (2011f:53063)}},
   doi={10.1007/s00023-010-0024-9},
}

\bib{MR0240748}{article}{
   author={Trudinger, N. S.},
   title={Remarks concerning the conformal deformation of Riemannian
   structures on compact manifolds},
   journal={Ann. Scuola Norm. Sup. Pisa (3)},
   volume={22},
   date={1968},
   pages={265--274},
   review={\MR{0240748 (39 \#2093)}},
}

\bib{MR1738176}{article}{
   author={Viaclovsky, J. A.},
   title={Conformal geometry, contact geometry, and the calculus of
   variations},
   journal={Duke Math. J.},
   volume={101},
   date={2000},
   number={2},
   pages={283--316},
   issn={0012-7094},
   review={\MR{1738176 (2001b:53038)}},
   doi={10.1215/S0012-7094-00-10127-5},
}

\bib{MR0125546}{article}{
   author={Yamabe, H.},
   title={On a deformation of Riemannian structures on compact manifolds},
   journal={Osaka Math. J.},
   volume={12},
   date={1960},
   pages={21--37},
   review={\MR{0125546 (23 \#A2847)}},
}

\end{biblist}
  \end{bibdiv}
 \end{document}